\documentclass{amsart}
\pdfoutput=1

\newcommand{\pres}[2]{\bigl\langle #1\:|\:#2 \bigr\rangle}
\newcommand{\Mpres}[2]{\operatorname{Mon}\bigl\langle #1\:|\:#2 \bigr\rangle}
\newcommand{\Mgen}[1]{\operatorname{Mon}\bigl\langle #1 \bigr\rangle}
\newcommand{\Ggen}[1]{\operatorname{Gp}\bigl\langle #1 \bigr\rangle}
\newcommand{\Spres}[2]{\operatorname{Sgp}\bigl\langle #1\:|\:#2 \bigr\rangle}
\newcommand{\Gpres}[2]{\operatorname{Gp}\bigl\langle #1\:|\:#2 \bigr\rangle}
\newcommand{\Ipres}[2]{\operatorname{Inv}\bigl\langle #1\:|\:#2 \bigr\rangle}

\newcommand{\M}{R_{m,n}}
\newcommand{\gl}{\mathcal{L}}

\newcommand{\pref}[1]{\operatorname{Pref}( #1 )}

\usepackage[leqno]{amsmath}
\usepackage[utf8]{inputenc}
\usepackage[english]{babel}

\usepackage{amssymb}
\usepackage{amsthm}
\usepackage{faktor}
\usepackage{mathtools}
\usepackage{amsfonts}
\usepackage{enumerate}
\usepackage{enumitem}
\usepackage{multicol}
\usepackage{float}
\usepackage{mathrsfs}
\usepackage{footnote}

\usepackage{tikz}
\usepackage{tikz-cd}
\usetikzlibrary{arrows,automata}
\usetikzlibrary{decorations.markings}
\usetikzlibrary{decorations.markings,arrows}
\usetikzlibrary{arrows,quotes}
\usetikzlibrary{arrows.meta}
\usetikzlibrary{lindenmayersystems,arrows.meta}
\usetikzlibrary{calc}

\tikzset{%
  >={Latex[width=2mm,length=2mm]},
            base/.style = {rectangle, rounded corners, draw=black,
                           minimum width=2cm, minimum height=0.8cm,
                           text centered, font=\sffamily},
  open/.style = {base, fill=white!30},
       undec/.style = {base, fill=red!30},
    dec/.style = {base, fill=green!30},
         process/.style = {base, minimum width=2.5cm, fill=orange!15,
                           font=\ttfamily},
}

\allowdisplaybreaks

\newcommand{\N}{\mathbb{N}}
\newcommand{\Z}{\mathbb{Z}}

\theoremstyle{plain}
\newtheorem{theorem}{Theorem}[section]
\newtheorem{prop}[theorem]{Proposition}

\newtheorem{lemma}[theorem]{Lemma}
\newtheorem{cor}[theorem]{Corollary}

\theoremstyle{definition}

\newtheorem{example}[theorem]{Example}
\newtheorem{question}[theorem]{Question}
\newtheorem{remark}[theorem]{Remark}
\newtheorem{construction}[theorem]{Construction}

\title[Membership problems for one-relator groups and monoids]{Membership problems for positive one-relator groups \\ and one-relation monoids}
\author{Islam Foniqi}

\address{Islam Foniqi, School of Mathematics, University of East Anglia, Norwich NR4 7TJ, England, UK}

\email{I.Foniqi@uea.ac.uk}

\author{Robert D. Gray }

\address{Robert D. Gray, School of Mathematics, University of East Anglia, Norwich NR4 7TJ, England, UK}

\email{Robert.D.Gray@uea.ac.uk}

\author{Carl-Fredrik Nyberg-Brodda}

\address{Carl-Fredrik Nyberg-Brodda, 
Korea Institute for Advanced Study (KIAS), Seoul, South Korea}

\email{cfnb@kias.re.kr}

\thanks{\\ The research of the first two named authors was supported by the EPSRC Fellowship grant EP/V032003/1 ‘Algorithmic, topological and geometric aspects of infinite groups, monoids and inverse semigroups’.}

\dedicatory{Dedicated to the memory of V.\ S.\ Guba ({1962}--{2023}).}


\makeatletter
\newcommand{\leqnomode}{\tagsleft@true\let\veqno\@@leqno}
\newcommand{\reqnomode}{\tagsleft@false\let\veqno\@@eqno}
\makeatother
\begin{document}
	
\begin{abstract}
Motivated by approaches to the word problem for one-relation monoids arising from work of Adian and Oganesian (1987), Guba (1997), and Ivanov, Margolis and Meakin (2001), 
we study the submonoid and rational subset membership problems in one-relation monoids and in positive one-relator groups. We give the first known examples of positive one-relator groups with undecidable submonoid membership problem, and apply this to give the first known examples of one-relation monoids with undecidable submonoid membership problem. We construct several infinite families of one-relation monoids with undecidable submonoid membership problem, including examples that are defined by relations of the form $w=1$ but which are not groups, and examples defined by relations of the form $u=v$ where both of $u$ and $v$ are non-empty. As a consequence we obtain a classification of the right-angled Artin groups that can arise as subgroups of one-relation monoids.  We also give examples of monoids with a single defining relation of the form $aUb = a$, and examples of the form $aUb=aVa$, with undecidable rational subset membership problem. 
We give a one-relator group defined by a freely reduced word of the form $uv^{-1}$ with $u, v$ positive words, in which the prefix membership problem is undecidable. Finally, we prove the existence of a special two-relator inverse monoid with undecidable word problem, and in which both the relators are positive words. As a corollary, we also find a positive two-relator group with undecidable prefix membership problem. In proving these results, we introduce new methods for proving undecidability of the rational subset membership problem in monoids and groups, including by finding suitable embeddings of certain trace monoids.
\end{abstract}
\maketitle

\vspace{-6mm}

\section{Introduction and summary of results}
Central among algorithmic problems in combinatorial algebra is the word problem which, given an algebraic structure defined by generators and relations, asks whether there is an algorithm which takes two expressions over the generators and decides whether they represent the same element. The word problem for finitely presented semigroups was proved undecidable by Markov and, independently, Post, in 1947. This was subsequently improved to the undecidability of the word problem in finitely presented cancellative semigroups by Turing \cite{Turing1950} resp.\ groups by Novikov in 1952 and, independently, Boone and Britton in 1958. In spite of, and of course unaware of, these general impossibilities, Magnus \cite{Magnus1932} had proved in 1932 that the word problem is decidable for all groups with only a single defining relation; such groups are now called \textit{one-relator groups}. By contrast, the word problem for monoids with one defining relation -- \textit{one-relation monoids} -- remains a tantalizing open problem, in spite of over a century of investigations; see \cite{NybergBrodda2021a} for a recent survey of the problem.

The majority of the results on the word problem for one-relation monoids have been focussed on trying to obtain a positive solution. The problem has now been solved positively in several cases. For example, Adian \cite{Adian1966} proved that the word problem is decidable for all \textit{special} one-relation monoids, being those admitting a presentation of the form $M = \Mpres{A}{w=1}$, and results of Adian \& Oganesian \cite{Adian1978, Adian1987} 
show that the word problem for a given $\Mpres{A}{u = v}$ can be reduced to the word problem for a one-relation monoid of the form  
$\Mpres{a,b}{bUa = aV a}$ or $\Mpres{a,b}{bUa = a}$. 
In both of these cases the word problem remains open. 

There are several important reduction results in the literature that relate the word problem in one-relation monoids to other natural decision problems in one-relator groups, one-relation monoids, and also in a class that lies between these two called inverse monoids. As we shall explain in more detail below, these reduction results divide into three interrelated approaches to the word problem for one-relation monoids, namely:  
(i) results of Ivanov, Margolis and Meakin \cite{Ivanov2001} that give a reduction to the word problem in one-relator inverse monoids, 
(ii) results of Guba \cite{Guba1997} that give a reduction 
to the submonoid membership problem in positive one-relator groups, 
and 
(iii) results of Adian and Oganesian 
\cite{Adian1966, Adian1976, Adian1987}
that give a reduction to the problem of deciding membership in principal right ideals of certain one-relation monoids.   
Here a one-relator group is \textit{positive} if it admits a presentation $\Gpres{A}{r=1}$ where no inverse symbol appears in $r$. Such groups were studied by e.g. Baumslag \cite{Baumslag1971}, as well as by Perrin \& Schupp \cite{Perrin1984}, who proved that a one-relator group is positive if and only if it is a one-relation monoid. 
There are also natural connections between the three approaches above. 
For example, Guba's reduction may alternatively be expressed as a question asking for membership in the submonoid of a one-relator group with defining relation of the form $uv^{-1}=1$ generated by the prefixes of the defining relation, where $u, v$ are both positive words. This \emph{prefix membership problem} for one-relator groups also arises naturally in the work of  
Ivanov, Margolis and Meakin \cite{Ivanov2001} where for cyclically reduced relator words they show that word problem for the inverse monoid reduces to the prefix membership problem for the group. 
In the reduction result (iii) the principal right ideals of one-relation monoids will not typically be finitely generated submonoids, but they are examples of rational subsets of the monoid. Hence one consequence of the reduction result (iii) of Adian and Oganesian is that a necessary step for constructing one-relation monoids with undecidable word problem is to first construct examples in which there are rational subsets in which membership is undecidable. 
This provides a connection between this approach and the approach of Guba to the word problem. 
Indeed, since by \cite{Perrin1984} every positive one-relator group is in particular a one-relation monoid,
the study of the submonoid membership problem for one-relation monoids has as a special case the submonoid membership problem for positive one-relator groups. 
Hence both of these approaches lie within the broader study of decidability of membership in rational subsets of one-relation monoids.  
These three approaches to the word problem for one-relation monoids with their various interrelations are summarized in the diagram in Figure~\ref{fig_implications}.
In addition to the motivation for their study coming from the connection with the word problem for one-relation monoids, the decision problems listed there, e.g. submonoid membership problem for one-relation monoids, are also natural questions to study in their own right.

\begin{figure}
\centering
    \resizebox{0.9\textwidth}{!}{%
\begin{tikzpicture}[node distance=1.8cm,
    every node/.style={fill=white,minimum height=1.5cm}, align=center, implies/.style={line width=0.75pt, double,double equal sign distance,-implies}]

  \newcommand\xsep{5}
  \newcommand\ysep{3}  

\node[text width=1cm]  (a) at (2.5,0.6){\tiny{Guba (1997)}};
\node[text width=1.5cm]  (b) at (2.5,-5){\tiny{Ivanov, Margolis, Meakin (2001)}};
\node[text width=1.8cm] (c) at (11.1,-4.6){\tiny{Ivanov, Margolis, Meakin (2001)}};
\node[text width=2.5cm] (d) at (11.2,1.6){\tiny{Adian (1967, 1976), Adian \& Oganesian (1987)
}};
\draw[rounded corners] (8, 0.9) rectangle (12, -3.9) {}; 
\draw[implies] (7.8,0) -- (6.9,0);

  \node (03)[open,minimum height=1.5cm,minimum width=3.5cm]  at (2*\xsep,\ysep){Principal right ideal MP in \\  $\Mpres{a,b}{bUa=a}$ \& \\ $\Mpres{a,b}{bUa=aVa}$};
 
  \node (02)[undec,minimum height=1.5cm,minimum width=3.5cm]  at (2*\xsep,2*\ysep){Rat Subset MP  in \\  $\Mpres{a,b}{bUa=a}$ \& \\ $\Mpres{a,b}{bUa=aVa}$ \\ \footnotesize{(Theorem~\ref{thm_UndecRatSubsetGubaType}, Remark~\ref{rem_MoreNonSubspecialExamples})} };

  \node (11)[open,minimum height=1.5cm,minimum width=3.5cm]  at (0,0){PMP for \\ $\Gpres{A}{cVdc^{-1}=1}$ \\ $V$ positive};
  \node (12)[open,minimum height=1.5cm,minimum width=3.5cm]   at (\xsep,0){WP for \\ $\Mpres{a,b}{bUa=a}$};
  \node (13)[open,minimum height=1.5cm,minimum width=3.5cm]   at (2*\xsep,0){WP for \\ $\Mpres{a,b}{bUa=a}$ \& \\ $\Mpres{a,b}{bUa=aVa}$};
  
  \node (21)[undec, minimum height=1.5cm,minimum width=3.5cm]   at (0,-\ysep){PMP for \\ $\Gpres{A}{uv^{-1}=1}$, \\ $uv^{-1}$ reduced \\ \footnotesize{(Theorem~\ref{Thm:exists-quasi-positive-undec-prefix})} };
  \node (new)[undec, minimum height=1.5cm,minimum width=3.5cm]   at (\xsep,-\ysep){WP for \\ $\Ipres{A}{u=1, v=1}$ \\ $u,v \in A^+$ positive \\ \footnotesize{(Theorem~\ref{thm_2RelatorInverseUndecWP}, Corollary~\ref{cor_TwoRelator})}};
  \node (22)[undec, minimum height=1.5cm,minimum width=3.5cm]   at (0,\ysep){Submonoid MP for \\ $\Gpres{A}{w=1}$ \\ $w \in A^+$ positive \\ \footnotesize{(Theorem~\ref{thm_mainPositive})}};
  \node (23)[open,minimum height=1.5cm,minimum width=3.5cm]   at (2*\xsep,-\ysep){WP for \\ $\Mpres{A}{u=v}$};
  
  \node (31)[open,minimum height=1.5cm,minimum width=3.5cm]   at (0,-2*\ysep){PMP for \\ $\Gpres{A}{uv^{-1}=1}$ \\ $uv^{-1}$ cyc. reduced};
  \node (32)[open,minimum height=1.5cm,minimum width=3.5cm]   at (\xsep,-2*\ysep){WP for \\ $\Ipres{A}{uv^{-1}=1}$ \\ $uv^{-1}$ cyc. reduced};
  \node (33)[open,minimum height=1.5cm,minimum width=3.5cm]   at (2*\xsep,-2*\ysep){WP for \\ $\Ipres{A}{uv^{-1}=1}$ \\ $uv^{-1}$ reduced};

 \node (252)[undec, minimum height=1.5cm,minimum width=3.5cm]   at (\xsep,\ysep){Submonoid MP for \\ $\Mpres{A}{u=v}$ 
 \\ \footnotesize{(Corollary~\ref{Cor:Exist-special-undecidable}} 
\\ \footnotesize{Prop.~\ref{prop_specialmonoidexamples} \& \ref{prop_subspecial}, Ex.~\ref{ex_compressSpecial})}};

  \node (new0)[open,minimum height=1.5cm,minimum width=3.5cm]  at (0,2*\ysep){Submonoid MP for \\ all surface groups};

  \node (new2)[undec,minimum height=1.5cm,minimum width=3.5cm]  at (\xsep,2*\ysep){Rat Subset MP  in \\  $\Mpres{A}{u=v}$ };


  \draw[implies] ($(new0)!2cm!(22)$)  --  ($(22)!2cm!(new0)$) ;

  \draw[implies] ($(new2)!2cm!(02)$)  -- ($(02)!2cm!(new2)$)  ;
  \draw[implies] ($(252)!2cm!(new2)$) -- ($(new2)!2cm!(252)$) ;

  \draw[implies] ($(new)!2cm!(33)$)  -- ($(33)!2cm!(new)$)  ;

  \draw[implies]  ($(13)!2cm!(03)$)  -- ($(03)!2cm!(13)$) ;

  \draw[implies] ($(03)!2cm!(02)$) -- ($(02)!2cm!(03)$) ;
 
  \draw[implies]  ($(252)!2cm!(22)$) --  ($(22)!2cm!(252)$) ;
  \draw[implies] ($(11)!2cm!(22)$) -- ($(22)!2cm!(11)$);
 
  \draw[implies] ($(12)!3cm!(11)$) -- ($(11)!3cm!(12)$);

  \draw[implies] ($(11)!2cm!(21)$) -- ($(21)!2cm!(11)$);
  \draw[implies] ($(13)!1.6cm!(23)$) -- ($(23)!2cm!(13)$);
  \draw[implies] ($(23)!1.5cm!(13)$) -- ($(13)!2cm!(23)$);
  \draw[implies] ($(23)!2cm!(33)$)  -- ($(33)!2cm!(23)$) ;  
  \draw[implies] ($(11)!2cm!(21)$) -- ($(21)!2cm!(11)$);
  \draw[implies] ($(32)!3cm!(33)$) -- ($(33)!3cm!(32)$);
  \draw[implies] ($(31)!2cm!(21)$) -- ($(21)!2cm!(31)$) ;
  \draw[implies] ($(32)!3cm!(31)$) -- ($(31)!3cm!(32)$);
  \end{tikzpicture}
}
\label{fig_implications}
\caption{
A summary of the main results of this article and how they relate to the three approaches to the word problem 
for one-relation monoids 
given by reduction results of 
(i) Ivanov, Margolis and Meakin \cite{Ivanov2001}, 
(ii) Guba \cite{Guba1997}, and  
(iii) Adian and Oganesian \cite{Adian1966, Adian1976, Adian1987}. 
The arrows indicate implication of decidability. 
The problems in red are all proved to be undecidable in this article in the results listed in the corresponding boxes.  
The problems in white boxes are all open. 
\vspace{-5mm}
}
\end{figure}

The connections explained above have led to extensive research and numerous positive decidability results have been obtained for special cases of these problems; see for example  
\cite{Guba1997, Ivanov2001, Jackson2001, Jackson2002, Arye2012, MMSunik2005, Meakin2007}.  
Until recently all of the focus of this work has been on showing that various special cases of these problems are decidable. 
However, several recent striking undecidability results 
in this area have for the first time brought into question the view that our attention should only be focussed on seeking positive solutions to such problems. 
First, Gray \cite{gray2020undecidability} proved the existence of a special one-relator inverse monoid $\Ipres{A}{w=1}$ with undecidable word problem and at the same time proved that there are one-relator groups with undecidable submonoid membership problem.  Then, Dolinka \& Gray \cite{dolinka2021new} went on to prove the existence of a one-relator group $\Gpres{B}{r=1}$ with undecidable prefix membership problem (where $r$ is a reduced word). Given the reduction results of Guba \cite{Guba1997} and Ivanov, Margolis, Meakin \cite{Ivanov2001} discussed above, if any of these problems had been decidable it would have resolved positively the word problem for one-relation monoids either in general, or for one of the two main open cases.  

These recent undecidability results give the first serious indication that the word problem for one-relation monoids could in fact be undecidable.  
If it is undecidable then, of course, all of the other reduction results mentioned above must also have negative answers. With that viewpoint in mind, the main goal of the current paper is to present a collection of new undecidability results all of the type that if they had been decidable then it would have solved positively the word problem for one-relation monoids (either in general, or in one of the two main open cases).  
For this we will conduct a detailed study of the rational subset, and submonoid, membership problems in positive one-relator groups, and in one-relation monoids. 
We will also introduce new tools for proving undecidability results of this kind.

Figure~\ref{fig_implications} gives a summary of the main undecidability results proved in this article and how they relate to the three approaches to the word problem discussed above. 
We shall now explain in more detail the reduction results discussed above, and summarized in Figure~\ref{fig_implications}, and in each case give an overview of the undecidability results related to them that we shall prove in this paper.

The word problem for one-relation monoids of the form $\Mpres{a,b}{bUa=a}$ remains open. We shall call these the \emph{monadic} one-relation monoids. Important work of Guba \cite[Corollary~2.1]{Guba1997} implies that the word problem for one-relation monoids in this class reduces to the submonoid membership problem for positive one-relator groups.
In more detail, it follows from the results of Guba \cite{Guba1997} that for every one-relation monoid of the form $M = \Mpres{a,b}{bQa=a}$, with $a$ and $b$ distinct, there exists a group defined by a presentation of the form $G = \Gpres{a,b,C}{aUba^{-1}=1}$, where $C$ is a finite set of new generators and $U$ is a positive word over $\{a,b\} \cup C$, such that if $G$ has decidable prefix membership problem then $M$ has decidable word problem. In fact, Guba proves the equivalent dual result reducing the problem to the suffix membership problem in $\Gpres{a,b,C}{a^{-1}bUa=1}$.
Note that $\Gpres{a,b,C}{abUa^{-1}=1}$ is a positive one-relator group, as it is isomorphic to $\Gpres{a,b,C}{bU=1}$. However, in general, the prefix monoid of $\Gpres{a,b,C}{abUa^{-1}=1}$ and of $\Gpres{a,b,C}{bU=1}$ will not be the same and, related to this, decidability of the prefix membership problem depends on the choice of presentation for a group rather than just its isomorphism type.
Clearly if the positive one-relator group $G$ has decidable submonoid membership problem, then in particular we can decide membership in the prefix monoid. Hence this reduction result of Guba motivates motivates the question of whether the submonoid membership problem is decidable for positive one-relator groups. 
Further motivation for this question comes from the fact that the submonoid membership problem is decidable for all surface groups if and only if it is decidable in the positive one-relator group with defining relation $a^2b^2c^2=1$; see Subsection~\ref{subsec_surface} for further discussion of this important open problem.  
As mentioned above, it was only recently discovered in \cite{gray2020undecidability} that there exist one-relator groups that contain finitely generated submonoids in which membership is undecidable. The first example of such a one-relator group was constructed by Gray \cite{gray2020undecidability}, and has the single defining relation $abba = baab$. Recently Nyberg--Brodda \cite{nyberg2022diophantine} proved that the one-relator groups $\Gpres{a,b}{a^mb^n = b^na^m}$ have undecidable submonoid membership problem for every $m, n \geq 2$. 
However, it may be shown (see below) that none of these one-relator groups admits a one-relator presentation with a positive defining relator word, that is, none of them are positive one-relator groups. The starting point for the work in this paper is to build on the examples from \cite{nyberg2022diophantine} to obtain positive one-relator groups with undecidable word problem by allowing $m$ and $n$ to vary through all integer values. 
As with previous examples, this is achieved by  
showing that the right-angled Artin group $A(P_4)$ of the path with four vertices embeds in the group, and then appealing to a theorem of Lohrey and Steinberg \cite[Theorem~2]{Lohrey2008}.  
However, what makes doing this more difficult than in the cases considered in \cite{nyberg2022diophantine} is that the embedded copies of $A(P_4)$ that we find in the positive one-relator groups are no longer subgroups of finite index, so Reidemeister--Schreier rewriting methods alone are not sufficient to obtain the result.
This gives the first main result of this paper (Theorem~\ref{thm_mainPositive}) where we exhibit an infinite family of positive one-relator groups all with undecidable submonoid membership problem.

In another direction, Guba's reduction motivates the question of whether the prefix membership problem is decidable for all one-relator groups of the form  $G = \Gpres{a,b,C}{aUba^{-1} = 1}$ where $U$ is a positive word. Here the defining relator is a reduced word of the form $uv^{-1}$ where $u$ and $v$ are both positive words. For words of this form we shall call the corresponding one-relator groups $\Gpres{A}{uv^{-1}=1}$ \emph{quasi-positive}. Quasi-positive one-relator groups have received attention in the literature in the study of the word problem for the related class of inverse monoids motivated by results of Ivanov, Margolis and Meakin \cite{Ivanov2001} discussed above and illustrated in the 
implications in the bottom two lines of Figure~\ref{fig_implications}. The prefix membership problem for various families of quasi-positive one-relator groups has been solved positively by Margolis, Meakin, and \v{S}uni\'{k} \cite[Corollary~2.6]{MMSunik2005}, and some cases of the word problem for the corresponding class of inverse monoids have been resolved by Inam \cite{Inam2017}. 
This connects more generally with the study of groups and inverse monoids defined by, so-called, Adian type presentations; see \cite{ Stallings1987, Inam2017b}.
In Section~\ref{Sec:prefix-membership-problem} we 
prove Theorem~\ref{Thm:exists-quasi-positive-undec-prefix} showing that that there is a one-relator group of the form $\Gpres{A}{uv^{-1}=1}$, where $u, v$ are positive words and $uv^{-1}$ is a reduced word, with undecidable prefix membership problem. 
This is the first known example of a quasi-positive one-relator group with undecidable prefix membership problem.

Our other main motivation for investigating membership problems in positive one-relator groups was the connection with the open question of whether one-relation monoids have decidable submonoid membership problem. When Magnus solved the word problem for one-relation groups he actually proved a more general result: membership in Magnus subgroups\footnote{These are subgroups generated by a subset of the generating set omitting a generator that appears in the defining relator word. Magnus called this specific membership problem the \textit{erweitertes Identitätsproblem}, i.e. ``extended word problem''.} is decidable. However, the general subgroup membership problem (also called generalized word problem) remains an important open problem for one-relator groups; see \cite[Problem~18]{BooneWP}. The analogue of this question for monoids asks whether one-relation monoids have decidable submonoid membership problem. This problem has been shown to be decidable in several examples and families of one-relation monoids; see \cite{Jackson2001, Jackson2002, Kambites2009, NybergBrodda2022a, Render2009}.  
Since by \cite{Perrin1984} a one-relator group is a one-relation monoid if and only if it is a positive one-relator group, the positive one-relator groups with undecidable submonoid membership problem that we give in this paper (Theorem~\ref{thm_mainPositive}) give the first known examples of one-relation monoids with undecidable submonoid membership problem (Corollary~\ref{Cor:Exist-special-undecidable}). Building on this, we shall construct several infinite families of one-relation monoids with undecidable submonoid membership problem including examples that are groups, examples that are defined by relations of the form $w=1$ but are not groups, and examples defined by relations of the form $u=v$ where neither $u$ nor $v$ is equal to $1$, so these monoids have trivial groups of units; see Propositions~\ref{prop_specialmonoidexamples} \& \ref{prop_subspecial}, and Example~\ref{ex_compressSpecial}.  As part of this, we also obtain a classification of exactly which right-angled Artin groups can appear as subgroups of one-relation monoids; see Theorem~\ref{thm_RAAGsEmbedding}.

Generalizing the submonoid membership problem is the \textit{rational subset membership problem}, which asks for an algorithm to decide membership in the image of an arbitrary regular language. In groups, this problem has been surveyed by Lohrey \cite{Lohrey2015}. The rational subset membership problem in one-relation monoids has also seen some study. Kambites' work \cite{Kambites2009} on small overlap conditions can be used to show that almost all one-relation monoids, in a  suitably defined sense, have decidable rational subset membership problem (cf. also \cite[p. 338]{NybergBrodda2021a}). Furthermore, the rational subsets of the bicyclic monoid $\Mpres{b,c}{bc=1}$ have been fully described by Render \& Kambites \cite{Render2009}, and more generally the rational subset membership problem in any $\Mpres{A}{w=1}$ with virtually free group of units is decidable \cite{NybergBrodda2022a}.  
Results of Adian and Oganesian \cite{Adian1966, Adian1976, Adian1987} imply that if membership in principal right ideals is decidable in all monoids of the form $\Mpres{a,b}{bUa=a}$ and $\Mpres{a,b}{bUa=aVa}$ then the word problem for all one-relation monoids is decidable (see also the related general statement \cite[Lemma~3.1]{Guba1997}). The principal right ideals of these monoids will not typically be finitely generated submonoids, but they are rational subsets. Motivated by this, in Section~\ref{Sec:Guba-type-monoids} we extend our investigation to the study of the rational subset membership problem in these two classes of one-relation monoids. Further motivation for studying this problem for monoids of the form $\Mpres{a,b}{bUa=a}$ comes from the result \cite[Theorem~4.1]{Guba1997} relating the word problem in these monoids to the membership problem in both principal left and principal right ideals. 
In Theorem~\ref{thm_UndecRatSubsetGubaType} we give an infinite family of monoids of the form $\Mpres{a,b}{bUa=a}$ all with undecidable rational subset membership problem. Then in Remark~\ref{rem_MoreNonSubspecialExamples} we explain how these examples can be adapted to give examples of the form $\Mpres{a,b}{bUa=aVa}$ with the same property.  To prove these results it is necessary for us to introduce new techniques since the only groups that monoids in these classes can embed are trivial groups, and hence the usual approach of embedding $A(P_4)$ is not possible.  For this we prove a new general result, Theorem~\ref{thm_LClassEmbeddingTraceMonoidImproved}, which shows that a left-cancellative monoid has undecidable rational subset membership problem if it embeds a copy of the trace monoid $T(P_4)$ that is generated by a set of elements that are all related by Green's $\mathcal{L}$-relation. 
This result is in some ways surprising since any trace monoid itself, unlike right-angled Artin groups, necessarily has decidable rational subset membership problem, so the way in which the trace monoid embeds is crucial. As another corollary to this new general result for left-cancellative monoids, we deduce (Corollary~\ref{cor_T4EmbedsInGroup}) that any group containing the trace monoid $T(P_4)$ has undecidable rational subset membership problem, and then explore some applications of this to proving the undecidability of the rational subset membership problem in groups.  

In Section~7 we shall apply the results of the previous section to the word problem for special inverse monoids. It was proved in \cite{gray2020undecidability} that there are special one-relator inverse monoids $\Ipres{A}{w=1}$ with undecidable word problem. In all known examples the word $w$ is not a reduced word, and it remains an open problem whether the word problem is decidable in that case. This is an important question since, if it is, then by \cite[Theorem~2.2]{Ivanov2001} this would imply that all one-relation monoids have decidable word problem. In particular it is open whether there is a positive one-relator inverse monoid with undecidable word problem. Motivated by this, in Section~7 we explain how the word problem for one-relation monoids also reduces to the word problem for positive $2$-relator special inverse monoids, and then in Theorem~\ref{thm_PosTwoRelatorGeneral} we show the existence of a $2$-relator special inverse monoid with undecidable word problem and in which both defining words are positive words. We use this result to prove the existence of a positive two-relator group with undecidable prefix membership problem in Corollary~\ref{cor_TwoRelator}.

\section{Preliminaries}\label{Sec2-Preliminaries}

In this section we fix some notation and recall some background definitions and results from geometric and combinatorial group and monoid theory. Inverse monoids will only be considered later in Section~\ref{Sec:Positive-inverse-monoids}, and we defer giving the more involved notation and definitions for inverse monoids to that section. For additional background we refer the reader to 
\cite{Magnus1966, lyndon1977combinatorial} for combinatorial group theory, \cite{Howie1995} for monoid and semigroup theory, \cite[Chapter~1]{Ruskuc1995} for monoid presentations, and 
\cite{Lawson1998} for the theory of inverse monoids. 

\subsection*{Monoid and group presentations}

For a non-empty alphabet $A$ we use $A^*$ to denote the \emph{free monoid} of all words over $A$ including the empty word which we denote by $\varepsilon$.    
A monoid presentation is a pair $\Mpres{A}{R}$ where $A$ is an alphabet and $R$ is a subset of $A^* \times A^*$. The monoid defined by this presentation is the quotient $A^* / \sigma$ of the free monoid by the congruence $\sigma$ on $A^*$ generated by $R$. We usually write a defining relation $(u,v) \in R$ as $u=v$. 
Similarly when working with a fixed monoid presentation $\Mpres{A}{R}$ given any two words $\alpha,\beta \in A^*$ we write $\alpha=\beta$ to mean that $\alpha$ and $\beta$ are $\sigma$-related, that is, they represent the same element of the monoid defined by the presentation. 
We write $\alpha \equiv \beta$ to mean that $\alpha$ and $\beta$ are equal as words in the free monoid $A^*$.       
A monoid presentation is called \emph{special} if all the defining relations are of the form $w=1$. 

A group presentation is a pair $\Gpres{A}{R}$ where $A$ is an alphabet and $R$ is a subset of  $(A \cup A^{-1})^* \times (A \cup A^{-1})^*$, where $A^{-1} = \{a^{-1} : a \in A \} $ is disjoint from $A$.     
The group defined by this presentation is then the quotient of the free group $F_A$ on $A$ by the normal subgroup generated by the set of all $uv^{-1}$ for $(u,v) \in R$. As for monoids, when working with a fixed group presentation we write $\alpha = \beta$ to mean that the words represent the same element of the group, and we write the defining relations in the form $u=v$.      

A one-relator monoid (also called one-relation monoid) is one defined by a presentation of the form $\Mpres{A}{u=v}$. Similarly a one-relator group is one that is defined by a presentation of the form $\Gpres{A}{w=1}$.  

We often abuse terminology by talking about the group $\Gpres{A}{R}$ or the monoid $\Mpres{A}{R}$ where we mean the group or monoid defined by the given presentation.   

Given a subset $X$ of a monoid we use $\Mgen{X}$ to denote the submonoid generated by $X$, and similarly if $X$ is a subset of a group then $\Ggen{X}$ denotes the subgroup generated by $X$.

\subsection*{Submonoid and rational subset membership problems}
The set of all \emph{rational subsets} of a monoid $M$ is the smallest subset of the power set of $M$ which contains 
all finite subsets of $M$, and which is closed under union, product, and Kleene hull.
Here the Kleene hull of a subset $L$ of a monoid $M$ is just the submonoid of $M$ generated by $L$.    
Clearly every finitely generated submonoid of $M$ is a rational subset.  
If $M$ is finitely generated by a set $A$, and $\phi:A^* \rightarrow M$ is the corresponding canonical homomorphism, then a subset $L \subseteq M$ is rational if and only if $L = \phi(K)$ for some rational subset $K$ of $A^*$. 
Since by Kleene's theorem the rational subsets of $A^*$ are the same as those that can be recognised by a finite state automaton, in this case $K$ is the language defined by some finite state automaton, and $L$ is the set of all elements of $M$ represented by words in that language.

Let $M$ be a monoid finitely generated by a set $A$, and let $\phi:A^* \rightarrow M$ be the corresponding canonical homomorphism. 
The \emph{submonoid membership problem} for $M$ is the following decision problem 

\begin{itemize} 
\item{\textsc{Input}:} A finite set of words $\Delta \subseteq A^*$ and a word $w \in A^*$   
\item{\textsc{Question}:} $\phi(w) \in \phi(\Delta^*)$?   
  \end{itemize}
Observe that $\phi(\Delta^*)$ is equal to the submonoid of $M$ generated by $\phi(\Delta)$.     
The decidability of this problem is independent of the choice of finite generating set for the monoid. For a group $G$ with finite generating set $A$, the set $A \cup A^{-1}$ is a finite monoid generating set for $G$, and then the submonoid membership problem for $G$ is defined as above, where $G$ is the monoid generated by $A \cup A^{-1}$.   
Let $L(\mathcal{A})$ denote the language recognised by a finite automaton $\mathcal{A}$. 
Then the \emph{rational subset membership problem} for $M$ is the decision problem  
\begin{itemize} 
\item{\textsc{Input}:} A finite automaton $\mathcal{A}$ over the alphabet $A$ and a word $w \in A^*$   
\item{\textsc{Question}:} $\phi(w) \in \phi(L(\mathcal{A}))$? 
  \end{itemize}
Note that by the Kleene Theorem, the input to the rational subset membership problem could alternatively be taken to be a rational expression over the alphabet $A$.
As for the submonoid membership problem, the rational subset membership problem also applies to groups where we view the group as a monoid generated by $A \cup A^{-1}$. As every ﬁnitely generated submonoid is a rational subset (being precisely the Kleene hull of a finite set), decidability of the rational subset membership problem implies decidability of the submonoid membership problem.

A priori, it may seem natural to assume that the rational subset membership problem ought to be strictly harder than the submonoid membership problem. However, whether this is actually the case remains an open problem; the problems may be equivalent, and there are some reasons to believe that they may be (see \cite{Lohrey2008}). 

In the two decision problems above, the submonoid (resp. rational subset) is part of the input. 
The non-uniform analogues of these problems can also be studied where one considers a fixed finitely generated submonoid (or rational subset) and asks whether there is an algorithm deciding membership in that particular subset. 
In general for any subset $S$ of $M$ by the 
\emph{membership problem for $S$ within $M$} we mean the decision problem:
\begin{itemize} 
\item{\textsc{Input}:} 
A word $w \in A^*$   
\item{\textsc{Question}:} $\phi(w) \in S$? 
  \end{itemize}
Similarly we talk about the membership problem for $S$ within $G$ where $G$ is a finitely generated group. This non-uniform version is sometimes called the \textit{weak} membership problem, with the uniform being called \textit{strong}, e.g. by Mikhailova \cite{Mikhailova1958} in the context of the subgroup membership problem.

Decidability of these problems is preserved when taking substructures in the following sense. 
Let $M$ be a finitely generated monoid and let $T$ be a finitely generated submonoid of $M$. 
If $M$ has decidable submonoid (resp. rational subset) membership problem then so does $T$. 
In addition to this, for any subset $S$ of $T$, if 
the membership problem for $S$ within $M$ is decidable then      
the membership problem for $S$ within $T$ is also decidable.
See \cite[\S5]{Lohrey2015} for more details on how to prove closure properties like these.    

\subsection*{RAAGs and trace monoids}

For a finite simplicial graph $\Gamma$ 
with vertex set $V\Gamma$  
we use $A(\Gamma)$ to denote the right-angled Artin group (abbreviated to RAAG) determined by $\Gamma$, so $A(\Gamma)$ is the group defined by the presentation 
\[
\Gpres{V\Gamma}{xy=yx \mbox{ \ if and only if $x$ is adjacent to $y$ in $\Gamma$} }.   
\]
We will use $T(\Gamma)$ to denote the corresponding \emph{trace monoid} defined by  
\[
\Mpres{V\Gamma}{xy=yx \mbox{ \ if and only if $x$ is adjacent to $y$ in $\Gamma$} }.   
\]
We shall now record some facts from the theory of trace monoids and RAAGs that we need later on. For more comprehensive background on RAAGs and trace monoids we refer the reader to \cite{Charney2007, Diekert1990}. It was proved by Paris \cite{Paris2002} that for any graph $\Gamma$ the identity map on $V\Gamma$ induces an embedding of the trace monoid $T(\Gamma)$ into the corresponding RAAG $A(\Gamma)$.  

Much is known about the behaviour of rational subsets of trace monoids (see e.g. \cite{Diekert1990}). In particular, since trace monoids are defined by presentations where the defining relations are all length preserving (so called `homogeneous presentations') it follows the any trace monoid has decidable rational subset membership problem. By contrast, not every RAAG has decidable rational subset membership. Indeed, Lohrey \& Steinberg \cite{Lohrey2008} proved that a RAAG $A(\Gamma)$ has decidable submonoid membership problem if and only if it has decidable rational subset membership problem if and only if $\Gamma$ does not embed the path $P_4$ with four vertices, or the square $C_4$ with four vertices, as an induced subgraph. A complete characterization of RAAGs with decidable subgroup membership problem is not, however, known; it is, for example, unknown whether $A(C_5)$ has decidable subgroup membership problem.

\subsection*{Inverse monoid presentations}

We just give the essential definitions from combinatorial inverse semigroup theory that we need. We refer the reader to \cite{Meakin2020} and \cite{Lawson1998} for a more comprehensive treatment of the subject.  

An inverse monoid $M$ is a monoid with the property that for every $m \in M$ there is a unique element $m^{-1} \in M$ satisfying $mm^{-1}m=m$ and $m^{-1}mm^{-1}=m^{-1}$.   
For any alphabet $A$ the \emph{free inverse monoid} over $A$ is the monoid defined by the presentation   
\[
\Mpres{A \cup A^{-1}}{uu^{-1}u=u,\ uu^{-1}vv^{-1}=vv^{-1}uu^{-1}\ (u,v\in\overline{A}^\ast)}
\]
where $(a^{-1})^{-1} = a$ and $(a_1 \ldots a_k)^{-1} = a_k^{-1} \ldots a_1^{-1}$.  
We use $\mathrm{FI_A}$ to denote the free inverse monoid on $A$.   
An inverse monoid presentation is a pair $\Ipres{A}{R}$ where $A$ is an alphabet and the set of defining relations $R$ is a subset of $(A \cup A^{-1})^\ast \times (A \cup A^{-1})^\ast$. 
The presentation $\Ipres{A}{R}$ defines the inverse monoid $\mathrm{FI_A}/\rho$ where $\rho$ is the congruence on the free inverse monoid $\mathrm{FI_A}$ generated by $R$. 
By a \emph{special inverse monoid} we mean one defined by a presentation where all the defining relations in the presentation are of the form $r=1$.  
The maximal group image of $\Ipres{A}{R}$ is the group $\Gpres{A}{R}$ defined by the same presentation, and the identity map on $A$ defines a surjective homomorphism from the inverse monoid onto its maximal group image.

\section{Membership Problems in Positive One-relator Groups}\label{Sec:SMMP-in-positive-OR-groups}

In this section, we construct positive one-relator groups with undecidable submonoid membership problem. 
As discussed in the previous section, the following one-relator groups:
\[
\Gpres{a,b}{[ab, ba] = 1}, \quad \text{resp. } \quad \pres{a,b}{[a^m, b^n] = 1} \: (m, n \geq 2),
\]
were shown to have undecidable submonoid membership problem in work of Gray \cite{gray2020undecidability} and Nyberg-Brodda \cite{nyberg2022diophantine}, respectively. 
However, it is a consequence of Lyndon's identity theorem that none of these groups admits a positive one-relator presentation. Indeed, it follows from Lyndon's identity theorem that any one-relator group of the form $\Gpres{A}{[u, v] = 1}$ has second homology group $H_2(G; \mathbb{Z}) \cong \mathbb{Z}$, while on the other hand any positive one-relator group has $H_2(G; \mathbb{Z}) = 0$; see \cite[II.4, Example 3]{Brown1982}. 

For $m,n \in \N$, define a class of groups by the presentation:
\reqnomode
\begin{equation}\label{Eq:G_mn-definition}
G_{m,n} = \Gpres{x,y}{x^m y^n = y^n x^{-m}}.
\end{equation}

\begin{remark}\label{G_{m,1} is BS(m,-m)}
If $n=1$ then
$G_{m,1} = \Gpres{x, y}{y x^m y^{-1} = x^{-m}}$ is the unimodular Baumslag-Solitar group $\operatorname{BS}(m, -m)$ which 
is known (see e.g. \cite[paragraph preceding Prop.~3.9]{nyberg2022diophantine}) to be virtually a direct product of $\Z$ and a finite rank free group, and hence $G_{m,1}$ has has decidable rational subset membership problem.  
\end{remark}

The key property of the groups $G_{m,n}$ is the following.

\begin{lemma}\label{lem_Gmn} 
The group
$G_{m,n}$
has decidable submonoid membership problem if and only if $m =1$ or $n=1$.
Furthermore, if $m, n \geq 2$ then $G_{m,n}$ contains a fixed finitely generated submonoid in which membership is undecidable.   
\end{lemma}
\begin{proof}
Decidability for $n = 1$ is discussed in Remark \ref{G_{m,1} is BS(m,-m)}, while the case for $m=1$ will be covered in Remark \ref{G_{1,n} finite extension of a RAAG}. For $m,n \geq 2$, we will obtain an injection $i \colon A(P_4) \longrightarrow G_{m,n}$ in Lemma \ref{A(P_4) inside G_{m,n}}. In \cite{Lohrey2008} it is shown that $A(P_4)$ contains a fixed submonoid where membership is undecidable. So, if $m,n \geq 2$, then the submonoid membership problem is undecidable in $G_{m,n}$.
\end{proof}

Thus to prove the above lemma, we must prove Lemma~\ref{A(P_4) inside G_{m,n}}. We do this by a Reidemeister--Schreier rewriting procedure.

\begin{lemma}\label{2n index subgroup presentation}
The subgroup $K_{m,n} = \Ggen{y^{2n}, y^i x y^{-i} \text{ for $0 \leq i \leq 2n - 1$}}$ of $G_{m,n}$, has the following presentation:
\[
\Gpres{\beta,\, \alpha_i \text{ for $0 \leq i \leq 2n - 1$}}{\alpha_i^m \alpha_{i+n}^m=1,\, [\alpha_i^m, \beta]=1 \text{ for $0 \leq i \leq n - 1$}}
\]
where $\beta$ corresponds to $y^{2n}$, and $\alpha_i$ corresponds to $y^i x y^{-i}$ for all $0 \leq i \leq 2n - 1$.    
\end{lemma}
\begin{proof}
Note that $K = K_{m,n}$ is normal in $G = G_{m,n}$, as it is closed under conjugation by both $x$ and $y$, with quotient equal to:
\[
G/K = \Gpres{y}{y^{2n}=1} \cong \mathcal{C}_{2n} = \text{ cyclic group of order $2n$}.
\]
We proceed to find a presentation for $K$ using the Reidemeister-Schreier procedure (see \cite{lyndon1977combinatorial}) with 
\[
G = \Gpres{S}{R}, \text{ where } S = \{x,y\}, \text{ and } R = x^m y^n x^m y^{-n}.
\]
A Schreier transversal for $K$ in $G$ is
$T = \{\, y^i \mid 0 \leq i \leq 2n - 1 \,\}$, giving a set of generators for $K$ as $U = \{ts(\overline{ts})^{-1}\mid t\in T,\; s \in S,\; ts \not\in T \}$ where $\overline{w}$ is the representative of $w$ in~$T$. 

Any $t\in T$ can be written as $y^i$ for $0 \leq i \leq 2n-1$, so:
\[
ts(\overline{ts})^{-1} = y^{i}s(\overline{y^{i}s})^{-1} = 
\begin{cases}
    y^{i}x(\overline{y^{i}x})^{-1} & \text{if } s = x,\\
    y^{i}y(\overline{y^{i}y})^{-1} & \text{if } s = y.
\end{cases}
\]
As $x\in K$, one has $\overline{y^{i}x} = y^i$ for all $0 \leq i \leq 2n-1$. Also $\overline{y^{i}y} = y^{i+1}$ for $0 \leq i \leq 2n-2$ and $\overline{y^{2n-1}y} = 1$. One obtains $\beta = y^{2n}$ and $\alpha_i = y^ixy^{-i}$ for $0 \leq i \leq 2n-1$ as generators of~$K$; hence
\[
U = \{\, \beta = y^{2n},\, \alpha_i = y^ixy^{-i} \text{ for $0 \leq i \leq 2n-1$}\,\},
\]
gives a set of generators for $K$.
     
To get relations for $K$, rewrite $tRt^{-1}$ for $t\in T$ and $R = x^m y^n x^m y^{-n}$, using generators in~$U$.

Write any $t\in T$ as $y^i$ for some $0 \leq i \leq 2n-1$. One obtains:
\begin{align*}
tRt^{-1} = y^i(x^m y^n x^m y^{-n})y^{-i} & = 
    \begin{cases}
    (y^i x^m y^{-i}) (y^{i + n} x^m y^{-i - n}) & \text{if } 0 \leq i \leq n - 1 \\
    (y^i x^m y^{-i}) y^{2n} (y^{i - n} x^m y^{-i + n}) y^{-2n} & \text{if } n \leq i \leq 2n - 1
\end{cases}
\\
& = 
\begin{cases}
    \alpha_i^m \alpha_{i+n}^m & \text{if } 0 \leq i \leq n - 1, \\
    \alpha_i^m \beta \alpha_{i-n}^m \beta^{-1} & \text{if } n \leq i \leq 2n - 1.
\end{cases}
\end{align*}
The first relation gives $\alpha_{i+n}^m = \alpha_i^{-m}$ for any $0 \leq i \leq n - 1$; substituting it in the second line, one obtains the relation $[\alpha_i^m, \beta]$ for all $0 \leq i \leq n - 1$.

So, $
V = \{\, \alpha_i^m \alpha_{i+n}^m,\, [\alpha_i^m, \beta] \mid 0 \leq i \leq n - 1 \,\}$ gives the relations of $K$, as desired.
\end{proof}

\begin{remark}\label{G_{1,n} finite extension of a RAAG}
Note that for $m=1$, the subgroup $K_{1,n}$ is actually a RAAG, isomorphic to $\Z \times F_n$. In this case, the  group $G_{1,n}$ is a finite extension of $\Z \times F_n$, implying that it has decidable rational subset membership problem (so, a fortiori, also decidable submonoid membership problem), see \cite{Lohrey2008}).
\end{remark}

Section 3 of article \cite{nyberg2022diophantine}, treats a class of groups called {\it right-angled Baumslag-Solitar-Artin groups}, for which the problem of decidable submonoid membership is studied. 
One particular example is the family $\operatorname{B}(S_{n,m})$ with $n,m \in \Z$, where $m,n \geq 0$, defined as:
\begin{equation}\label{B(S_{n,m})}
\operatorname{B}(S_{n,m}) = \Gpres{d, c_i \text{ for $0 \leq i \leq n - 1$}}{[c_i^m, d]=1 \text{ for $0 \leq i \leq n - 1$}}.
\end{equation}
Another example is the family $B(P_{3,m})$ with $m \in \Z$ where $m \geq 0$, defined as:    
\begin{equation}\label{EqP3m}
\operatorname{B}(P_{3,m}) = \Gpres{w_0,w_1,w_2}{[w_0,w_2]=1, [w_0,w_1^m]=1}.
\end{equation}

\begin{lemma}\label{B(S_{m,n}) injects in G_{m,n}}
The group $\operatorname{B}(S_{n,m})$ injects in~$G_{m,n}$. In particular, the map 
\begin{equation}\label{Eq:Sigma-embedding}
\sigma \colon d \mapsto y^{2n}, c_i \mapsto y^ixy^{-i} \: (0 \leq i \leq n)
\end{equation}
defines an injective homomorphism $\sigma \colon \operatorname{B}(S_{n,m}) \to G_{m,n}$.
\end{lemma}
\begin{proof}
Recall from Lemma~\ref{2n index subgroup presentation} 
that the subgroup $K_{m,n} = \Ggen{y^{2n}, y^i x y^{-i} \text{ for $0 \leq i \leq 2n - 1$}}$ of $G_{m,n}$, has the following presentation:
\[
\Gpres{\beta,\, \alpha_i \text{ for $0 \leq i \leq 2n - 1$}}{\alpha_i^m \alpha_{i+n}^m=1,\, [\alpha_i^m, \beta]=1 \text{ for $0 \leq i \leq n - 1$}}
\]
where $\beta$ corresponds to $y^{2n}$, and $\alpha_i$ corresponds to $y^i x y^{-i}$ for all $0 \leq i \leq 2n - 1$.    
Let $K \cong K_{m,n} $ denote the group defined by the above presentation, let $B = \operatorname{B}(S_{n,m})$, and 
consider the two maps:
\begin{align*}
s: B \longrightarrow K\, & \text{ given by } s(d) = \beta,\, s(c_i) = \alpha_i\, \text{ for any } 0 \leq i \leq n - 1,\\
\rho: K \longrightarrow B\, & \text{ given by } \rho(\beta) = d,\, \rho(\alpha_i) = c_i,\, \rho(\alpha_{i+n}) = c_i^{-1}\, \text{ for any } 0 \leq i \leq n - 1.
\end{align*}
Obviously, both $s$ and $\rho$ induce well-defined homomorphisms, as they respect the relations. Moreover, one has $\rho\circ s = \text{id}_B$, which implies that $s$ is injective. 
It follows that there is an injective homomorphism  
$\sigma:\operatorname{B}(S_{n,m}) \hookrightarrow G_{m,n}$ 
that maps $d \mapsto y^{2n}$ and $c_i \mapsto y^i x y^{-i}$ for all $0 \leq i \leq n - 1$.     
\end{proof}

\begin{lemma}\label{A(P_4) inside G_{m,n}}
There exists an embedding $i\colon A(P_4) \hookrightarrow G_{m,n}$ 
from 
\[
A(P_4) = \Gpres{A, B, C, D}{AB = BA, BC = CB, CD = DC} 
\]
into $G_{m,n}$ given by 
\[
i(A) = y x^m y^{-1},\, i(B) = y^{2n},\, i(C) = x^m,\, i(D) = x y^{2n} x^{-1}.
\]
\end{lemma}

\begin{proof}
From the proof of Proposition 3.5 in \cite{nyberg2022diophantine} (see also Proposition 3.3 there), one obtains an embedding of $A(P_4)$ in $\operatorname{B}(S_{n,m})$ for $n, m \geq 2$. One such embedding is the following:
\[
j: \Gpres{A, B, C, D}{AB = BA, BC = CB, CD = DC} \hookrightarrow \operatorname{B}(S_{n,m}),
\]
given by $j(A) = c_1^m,\, j(B) = d,\, j(C) = c_0^m,\, j(D) = c_0dc_0^{-1}$. Indeed, to show that $j$ is an embedding, we note that by Lemma \ref{B(S_{m,n}) injects in G_{m,n}} the group $\operatorname{B}(S_{n,m})$ injects into $G_{m,n}$, where 
\[
\operatorname{B}(S_{n,m}) = \Gpres{d, c_i \text{ for $0 \leq i \leq n - 1$}}{[c_i^m, d]=1 \text{ for $0 \leq i \leq n - 1$}},
\]
with $d \mapsto y^{2n}$, and for all $0 \leq i \leq n - 1$ one has $c_i \mapsto y^ixy^{-i}$. In particular the RABSAG $\operatorname{B}(S_{2,m})$ (in the sense of \cite{nyberg2022diophantine}) defined by the graph
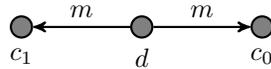
\[
\begin{tikzpicture}[>=stealth',thick,scale=0.8,el/.style = {inner sep=2pt, align=left, sloped}]%

\node (w0)[label=below:$d$][circle, draw, fill=black!50,
                        inner sep=0pt, minimum width=8pt] at (2,-1) {};
\node (w1)[label=below:$c_1$][circle, draw, fill=black!50,
                        inner sep=0pt, minimum width=8pt] at (0,-1) {};
\node (w2)[label=below:$c_0$][circle, draw, fill=black!50,
                        inner sep=0pt, minimum width=8pt] at (4,-1) {};
                        
\path[->](w0)  edge node[above]{$m$}         (w1);
\path[->](w0)  edge node[above]{$m$}         (w2);
\end{tikzpicture}
\]
injects in $\operatorname{B}(S_{n,m})$.  By \cite[Proposition~3.5]{nyberg2022diophantine} and its proof, the group $\operatorname{B}(P_{3,m})$ injects into $\operatorname{B}(S_{n,m})$, with one embedding being via the RABSAG defined by the graph
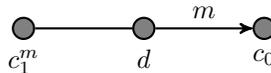
\[
\begin{tikzpicture}[>=stealth',thick,scale=0.8,el/.style = {inner sep=2pt, align=left, sloped}]%

\node (w0)[label=below:$d$][circle, draw, fill=black!50,
                        inner sep=0pt, minimum width=8pt] at (2,-1) {};
\node (w1)[label=below:$c_1^m$][circle, draw, fill=black!50,
                        inner sep=0pt, minimum width=8pt] at (0,-1) {};
\node (w2)[label=below:$c_0$][circle, draw, fill=black!50,
                        inner sep=0pt, minimum width=8pt] at (4,-1) {};
                        
\path[-](w0)  edge node[above]{}         (w1);
\path[->](w0)  edge node[above]{$m$}         (w2);
\end{tikzpicture}
\]      
Here, we have abused notation by letting the vertices of the graph $P_{3,m}$ above be labelled by their images in $\operatorname{B}(S_{n,m})$ under the prescribed embedding. That is, inside $\operatorname{B}(S_{n,m})$, defined as above, the elements $c_1^m, d$, and $c_0$ generate an isomorphic copy of $\operatorname{B}(P_{3,m})$
where, in terms of the presentation given in equation~\eqref{EqP3m}, this isomorphism is given by mapping $w_0$ to $d$, $w_2$ to $c_1^m$, and $w_1$ to $c_0$. 
Continuing, and using the analogous slight abuse of notation also for RAAGs, the RAAG defined by the graph 
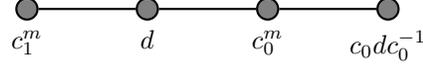
\[
\begin{tikzpicture}[>=stealth',thick,scale=0.8,el/.style = {inner sep=2pt, align=left, sloped}]%

\node (w0)[label=below:$c_1^m$][circle, draw, fill=black!50,
                        inner sep=0pt, minimum width=8pt] at (0,0) {};
\node (w1)[label=below:$d$][circle, draw, fill=black!50,
                        inner sep=0pt, minimum width=8pt] at (2,0) {};
\node (w2)[label=below:$c_0^m$][circle, draw, fill=black!50,
                        inner sep=0pt, minimum width=8pt] at (4,0) {};
\node (w3)[label=below:$c_0dc_0^{-1}$][circle, draw, fill=black!50,
                        inner sep=0pt, minimum width=8pt] at (6,0) {};
                        
\path[-](w0)  edge node[above]{}         (w1)
		(w1)  edge node[above]{}         (w2)
		(w2)  edge node[above]{}         (w3);
\end{tikzpicture}
\]
now embeds in our chosen copy of $\operatorname{B}(P_{3,m})$ by unpacking the proof of \cite[Proposition~3.3]{nyberg2022diophantine} together with the generalisation of that result explained in \cite{nyberg2022diophantine} in the paragraph immediately after the proof of \cite[Proposition~3.3]{nyberg2022diophantine}.  
That is, inside $\operatorname{B}(S_{n,m})$ as above, the elements $c_1^m, d, c_0^m$, and $c_0dc_0^{-1}$ generate a copy of $A(P_4)$. It follows that $j$ is indeed an embedding of $A(P_4)$ into $\operatorname{B}(S_{n,m})$.

Now the composition $i = \sigma \circ j$ where $\sigma$ is the injection $\sigma \colon \operatorname{B}(S_{n,m}) \hookrightarrow G_{m,n}$ defined by \eqref{Eq:Sigma-embedding}, from Lemma \ref{B(S_{m,n}) injects in G_{m,n}}, gives the desired embedding $i\colon A(P_4) \hookrightarrow G_{m,n}$.
\end{proof}

Putting it all together, we have now shown Lemma~\ref{lem_Gmn}. In particular, for every $m, n \geq 2$ the one-relator group $G_{m,n}$ has undecidable submonoid membership problem. Importantly, these groups also have the following property:

\begin{lemma}\label{lem_key} 
For every $m, n \geq 1$, the group $G_{m,n}$ is a positive one-relator group. 
\end{lemma}
\begin{proof}
Consider the relation $x^m y^n x^m y^{-n}$. We want to change it to an equivalent positive relation, so we introduce new generators $a,b$ with $y = a$ and $x = b a ^{n}$. Now we write the group relation in terms of $a$ and $b$ as:
\begin{align*}
(b a^{n})^m a^n (b a^{n})^m a^{-n} & = (b a^{n})^m a^n (b a^{n})^{m-1} (b a^{n}) a^{-n} \\
& = (b a^{n})^m a^n (b a^{n})^{m-1} b \\
 & = (b a^{n})^m (a^{n} b)^m.
\end{align*}
This way, we obtain another equivalent presentation for $G_{m,n}$ as:
\begin{equation}\label{Eq:Gmn-positive-presentation}
G_{m,n} \cong \Gpres{a,b}{(b a^n)^m(a^n b)^m = 1}.
\end{equation}
Thus $G_{m,n}$ is a positive one-relator group.
\end{proof}

We have proved the following, which is the main result of this section:

\begin{theorem}\label{thm_mainPositive} 
For all $m,n \geq 2$ the group 
\[
G_{m,n} \cong \Gpres{a,b}{(b a^n)^m(a^n b)^m = 1}
\]
is a positive one-relator group that contains a fixed finitely generated submonoid in which membership is undecidable. 
In particular, there are positive one-relator groups with undecidable submonoid membership problem.  
  \end{theorem}

Note that if $M_{m,n}$ denotes the monoid with the same defining relation as in \eqref{Eq:Gmn-positive-presentation}, then $b$ is invertible in $M_{m,n}$, being both left and right invertible by virtue of the defining relation. By cyclically permuting the letters $b$ from the ends of the defining relation, we similarly also conclude that $a$ is invertible. Hence $M_{m,n}$ is a group, so necessarily $M_{m,n} = G_{m,n}$. Thus, we conclude: 

\begin{cor}\label{Cor:Exist-special-undecidable}
There exists a one-relation monoid with undecidable submonoid membership problem. Furthermore, there exists such a monoid with a presentation of the form $\Mpres{A}{w=1}$. Additionally, one can construct such monoids which contain a fixed finitely generated submonoid in which membership is undecidable.  
\end{cor}

Using this, it is now not difficult to find the following classification of right-angled Artin subgroups of one-relation monoids:

\begin{theorem}\label{thm_RAAGsEmbedding} 
Let $\Gamma$ be a finite graph and let $A(\Gamma)$ be the right-angled Artin group that it defines. Then the following are equivalent. 
\begin{enumerate} 
\item[(i)] $\Gamma$ is a forest; 
\item[(ii)] $A(\Gamma)$ embeds into a one-relator group;  
\item[(iii)] $A(\Gamma)$ embeds into a positive one-relator group;  
\item[(iv)] $A(\Gamma)$ embeds into a one-relation monoid $\Mpres{A}{u=v}$;   
\item[(v)] $A(\Gamma)$ embeds into a one-relation monoid of the form $\Mpres{A}{w=1}$.   
  \end{enumerate} 
  \end{theorem}
\begin{proof} 
The equivalence of (i) and (ii) follows from \cite[Remark 2.3]{gray2020undecidability}.

(i) $\Rightarrow$ (iii): 
By \cite[Theorem 1.8]{Kim2013} if $\Gamma$ is a finite forest then $A(\Gamma)$ embeds in $A(P_4)$ which in turn embeds in a positive one-relator group by Lemma~\ref{A(P_4) inside G_{m,n}} and Lemma~\ref{lem_key} above. 

(iii) $\Rightarrow$ (v): 
This follows from the result of Perrin and Schupp \cite{Perrin1984} showing that any positive one-relator group admits a one-relation monoid presentation. 

(v) $\Rightarrow$ (iv) is trivial. 

(iv) $\Rightarrow$ (ii): 
If $A(\Gamma)$ embeds into $\Mpres{A}{u=v}$ then it must embed in a maximal subgroup of $\Mpres{A}{u=v}$ which by the results of \cite{Lallement1974} must itself be a one-relator group. Hence $A(\Gamma)$ embeds in a one-relator group.   
  \end{proof}

\subsection{Positive one-relator groups and surface groups}\label{subsec_surface}

The class of positive one-relator groups has received some attention in the literature. They were first studied by Baumslag \cite{Baumslag1971} in 1971, who proved that the intersection of all terms in the lower central series in any positive one-relator group is trivial. While Perrin \& Schupp \cite{Perrin1984} observed that not all positive one-relator groups are residually finite, Wise \cite{Wise2001} later studied the residual finiteness of positive one-relator groups, proving that any positive one-relator group with torsion is residually finite (this result was later sharpened by Wise \cite{Wise2021} by dropping the positivity condition). In general, classifying which positive one-relator groups have decidable submonoid membership problem seems difficult. It is even unknown whether every one-relator group with torsion has decidable submonoid membership problem \cite[Questions~20.68 \& 20.69]{Kourovka2020}. One-relator groups with torsion are hyperbolic by the B.~B.~Newman Spelling Theorem; thus, a natural question in line with this was implicitly asked by Ivanov, Margolis \& Meakin \cite[p. 110]{Ivanov2001}: 

\begin{question}\label{Quest:Surface-group-smmp}
Is the submonoid/rational subset membership problem decidable in every surface group? 
\end{question}

Here, a \textit{surface group} $G$ (of genus $g \geq 0$) is one which is the fundamental group of some compact $2$-manifold; thus either $G \cong \mathcal{N}_g$ or $G \cong \mathcal{S}_g$, where
\[
\mathcal{N}_g = \Gpres{a_1, \dots, a_g}{a_1^2 a_2^2 \cdots a_g^2 = 1} \:\: \text{and} \:\: \mathcal{S}_g = \Gpres{a_1, b_1, \dots, a_g, b_g}{\prod_{i=1}^g [a_i, b_i] = 1}.
\]
We remark that the sub\textit{group} membership problem in all surface groups is, by contrast, well-known to be decidable \cite{Scott1978}, but in general hyperbolic groups can even have undecidable subgroup membership problem by using the Rips construction \cite{Rips1982}. Note further that $\mathcal{N}_2 \cong \Gpres{a,b}{a^2b^2 = 1}$ and $\mathcal{S}_1 \cong \Z^2$ are virtually abelian, and hence have decidable rational subset membership problem \cite{Grunschlag1999}. All other surface groups are hyperbolic one-relator groups. If Gromov's famous Surface Subgroup Conjecture holds for one-relator groups (see e.g. \cite{Gordon2010}), then any non-virtually free one-relator group contains a subgroup $\mathcal{S}_g$ for some $g \geq 2$. Hence, understanding membership problems for surface groups can be seen as a key first step to understanding membership problems in hyperbolic one-relator groups. 

It is well-known (see e.g. \cite[\S4.3.7]{Stillwell1993} or \cite{Hoare1971}) that the group $\mathcal{N}_3$ contains a finite index copy of every hyperbolic surface group. Consequently, Question~\ref{Quest:Surface-group-smmp} is equivalent to:

\begin{question}\label{Quest:SMMP-in-N3}
Is the submonoid/rational subset membership problem decidable in the positive one-relator group $\mathcal{N}_3 = \Gpres{a,b,c}{a^2b^2c^2=1}$?
\end{question}

The word $a^2b^2c^2$ has length $6$. It is easy to see that any one-relator group $G = \Gpres{A}{r=1}$ with $|r| < 6$ is either free or isomorphic to a free product of a free group by a two-generated one-relator group $H = \Gpres{a,b}{r=1}$ with $|r| < 6$. There are only seven (non-free) isomorphism types of such $H$, namely: $C_2, C_3, C_4, C_5, \mathcal{N}_2, \Z^2$, and the torus knot group $\Gpres{a,b}{a^2b^3=1}$. All such groups have decidable rational subset membership problem (see \cite{nyberg2022diophantine}), and hence, as decidability of the rational subset membership problem is preserved by taking free products, we conclude that any $G$ as above (with $|r| < 6$) has decidable rational subset membership problem. Thus, the group $\mathcal{N}_3$ is the smallest (in terms of relator word length) candidate for a one-relator group with undecidable rational subset membership problem.

\section{Prefix Membership Problems}\label{Sec:prefix-membership-problem}

We say that a one-relator group $G$ is \textit{quasi-positive} if it is given by a presentation of the form $G = \Gpres{A}{uv^{-1}=1}$ where $u, v \in A^+$ are positive words, and $uv^{-1}$ is a reduced word. 
As explained in the introduction, and summarized in the diagram of implications in Figure~\ref{fig_implications}, both results of Guba \cite{Guba1997} and work of Ivanov, Margolis and Meakin \cite{Ivanov2001} motivate the study of the prefix membership problem for quasi-positive one-relator groups.   
The main result of this section is:
\begin{theorem}\label{Thm:exists-quasi-positive-undec-prefix}
There exists a one-relator group 
$G = \Gpres{A}{uv^{-1}=1}$, 
where $u, v \in A^+$ and $uv^{-1}$ is a reduced word, 
with an undecidable prefix membership problem. 
\end{theorem}

Theorem~\ref{Thm:exists-quasi-positive-undec-prefix} 
will be proved by encoding the submonoid membership problem for a positive one-relator group into the prefix membership problem for a quasi-positive one-relator group. The encoding uses a general construction that will also be used to establish other undecidability results in this paper e.g. for inverse monoids, so we will explain it here so that it can be applied in each instance that it is needed.  

\begin{construction}\label{construction_encoding}
Let $G$ be a positive one-relator group and let $Q$ be a finitely generated submonoid of $G$.     
Let $\Mpres{A}{q=1} \cong \Gpres{A}{q=1}$ be a one-relation monoid presentation for the group, which exists by \cite{Perrin1984} since $G$ is a positive one-relator group, and let $a$ denote the first letter of the word $q$. Let $Q$ be a finitely generated submonoid of $G$.  Let $X = \{w_1, \ldots, w_k \} \subseteq A^+$ be a set of positive words such that $Q = \Mgen{w_1, \ldots, w_k} \leq G$.  Such a set of positive words $X$ exists since  every element of $G$ can be expressed by a positive word over $A$ as this is true in the monoid $M$. For any $w \in A^+$ let $\overline{w} \in A^+$ such that $\overline{w} = w^{-1}$ in $G$, i.e. $w \overline{w} = \overline{w} w = 1$ in $G$.     Let $z_i$ be word obtained from $w_i$ after replacing $a$ with $tx$ for every letter $a$.  Similarly, let $\overline{z_i}$ be the word obtained by replacing $a$ with $tx$ in $\overline{w_i}$ for every occurrence of the letter $a$.  Let $r$ be the word obtained by replacing every occurrence of $a$ with $tx$ in the word $q$. Let $B = A \setminus \{a\}$. Then $r$, $z_i$ and $\overline{z_i}$ for $1 \leq i \leq k$ are all positive words over $B \cup \{x,t\}$, and $r$ begins with $tx$ since the first letter of $q$ is $a$. 
Write $r \equiv ts$ , i.e. making $s$ the positive word obtained by deleting the first letter of $r$. In particular $s$ begins with the letter $x$.                       

Using the data above we define a two-relator group presentation 
\[
H_{G,X} = \Gpres{B,x,t}{
r = 1, \quad 
t z_1 s t\overline{z_1} s  \ldots  s t z_k s t\overline{z_k} s = 1
}.
\]
For future reference, we also use $M_{G,X}$ to denote the corresponding inverse monoid presentation  
\[
M_{G,X} = \Ipres{B,x,t}{
r = 1, \quad 
t z_1 s t\overline{z_1} s  \ldots  s t z_k s t\overline{z_k} s = 1
}.
\]
\end{construction}

To establish Theorem~\ref{Thm:exists-quasi-positive-undec-prefix} we will first prove a general result which shows how the word problem in any finitely generated submonoid of a positive one-relator group can be encoded in the prefix membership problem in a quasi-positive one-relator group, and then we combine that with the examples from Section~\ref{Sec:SMMP-in-positive-OR-groups} to obtain Theorem~\ref{Thm:exists-quasi-positive-undec-prefix}. 
To prove this general result we will make use of a related general result (Theorem~\ref{thm_TwoRelator}) about the prefix membership problem for positive two-relator groups that will be proved in Section~\ref{Sec:Positive-inverse-monoids} below as an application of results proved there about the word problem for two-relator inverse monoids. 

\begin{theorem}\label{Thm:quasi-positive-prefix-reduces-to-membership-new-easy-proof}
Let $G$ be a positive one-relator group, and let $Q$ be any finitely generated submonoid of $G$. 
Then there exists a quasi-positive one-relator group $G'$ such that 
the membership problem for $Q$ in $G$ reduces to 
the prefix membership problem for $G'$.
Furthermore, $G'$ can be chosen such that $G' \cong G \ast \mathbb{Z}$. 
\end{theorem}\begin{proof}
Let 
\[
H_{G,X} = \Gpres{B,x,t}{
r = 1, \quad 
t z_1 s t\overline{z_1} s  \ldots  s t z_k s t\overline{z_k} s = 1
}
\]
be the positive two-relator group given by Construction~\ref{construction_encoding}.   
Set $H = H_{G,X}$, $Y = B \cup \{x,t\}$, $u=r$ and $v = t z_1 s t\overline{z_1} s  \ldots  s t z_k s t\overline{z_k} s$.  
It then follows from Theorem~\ref{thm_TwoRelator} and its proof that 
the membership problem for $Q$ in $G$ reduces to the prefix membership problem for $H$,  
that the identity map on $Y$ induces an isomorphism 
$\Gpres{Y}{u=1, v=1} \rightarrow \Gpres{Y}{u=1}$,
and that $H \cong G \ast \mathbb{Z}$.    
It follows that the identity map on $Y$ induces an isomorphism  
\[
\Gpres{Y}{u=1,v=1} \rightarrow \Gpres{Y}{vuv^{-1}=1}, 
\] 
and this group is isomorphic to $G \ast \mathbb{Z}$.
Now since $v=1$ in this group, the prefix monoids of   
$\Gpres{Y}{u=1,v=1}$ and $\Gpres{Y}{vuv^{-1}=1}$ are both generated by $\pref{u} \cup \pref{v}$. 
Hence the isomorphism induced by the identity map on $Y$ maps the prefix monoid of   
$\Gpres{Y}{u=1,v=1}$ bijectively to the prefix monoid of $\Gpres{Y}{vuv^{-1}=1}$. 
Therefore if $\Gpres{Y}{vuv^{-1}=1}$ has decidable prefix membership problem then so does 
$\Gpres{Y}{u=1,v=1}$ which in turn 
by Theorem~\ref{thm_TwoRelator} 
implies that the membership problem for $Q$ in $G$ is decidable. 
The result then follows by taking $G'$ to be the one-relator group with generating set $Y$ and defining relator the reduced form of $vuv^{-1}$.    
\end{proof}

Applying this general result with our examples from earlier gives: 

\begin{proof}[Proof of Theorem~\ref{Thm:exists-quasi-positive-undec-prefix}] By Theorem~\ref{thm_mainPositive} there exists a positive one-relator group $G$ with a fixed finitely generated submonoid $Q$ such that the membership problem for $Q$ in $G$ is undecidable. Hence, by Theorem~\ref{Thm:quasi-positive-prefix-reduces-to-membership-new-easy-proof} there is a quasi-positive one-relator group $G'$ such that the membership problem for $Q$ in $G$ reduces to the prefix membership problem for $G'$. Hence $G'$ is a quasi-positive one-relator group with undecidable prefix membership problem. 
\end{proof}

\section{Membership Problems in One-relation Monoids}\label{Sec:SMMP-in-proper-one-relation-monoids}
The subgroup membership problem (also called the generalized word problem) is open in general for one-relator groups. The analogous question for monoids asks whether all one-relation monoids have decidable submonoid membership problem. As well as being a natural question, another reason for studying the membership problem in submonoids, and more generally rational subsets, of one-relation monoids comes from the \textit{left} (resp. \textit{right}) \textit{divisibility} problem, i.e. the problem of deciding membership in the principal right (resp. left) ideal generated by a given element.
By a classical result of Adian \& Oganesian \cite[Corollary~3]{Adian1978}, decidability of the divisibility problems in one-relation monoids $\Mpres{A}{u=v}$ implies decidability of the word problem. Clearly, these principal one-sided ideals are rational subsets of the monoid. Furthermore, Guba \cite{Guba1997} proved that for one-relation monoids of the form $\Mpres{a,b}{bUa=a}$, the decidability of the membership problem in principal right ideals is \textit{equivalent} to the word problem. This motivates the study of the membership problem for rational subsets of one-relation monoids. As outlined in the introduction above, several examples and families of one-relation monoid have been shown to have decidable rational subset membership problem, and in some sense most of them do in the way that for a randomly chosen one-relation monoid this problem will be decidable \cite{Kambites2009}, cf. \cite[p. 338]{NybergBrodda2021a} for a discussion.

Since every positive one-relator group admits a one-relation monoid presentation, the main result of \S\ref{Sec:SMMP-in-positive-OR-groups} (Theorem~\ref{thm_mainPositive}) gave as a corollary (Corollary~\ref{Cor:Exist-special-undecidable}) the first known examples of one-relation monoids for which the submonoid membership problem is undecidable. Specifically we have shown that the one-relation monoid  
\[
M_{m,n} = \Mpres{a,b}{(b a^n)^m(a^n b)^m = 1}
\]
with $m, n \geq 1$  has decidable submonoid membership problem (and rational subset membership problem) if and only if $m=1$ or $n=1$. These are the first known examples of one-relation monoids with undecidable submonoid membership problem (and undecidable rational subset membership problem). Of course all of these monoids are in fact groups. Thus arises the question: are there one-relation monoids that are not groups and have undecidable submonoid, or rational subset, membership? More generally we have the following problem:  

\smallskip

\noindent \textbf{Problem:}
Classify the one-relation monoids $\Mpres{A}{u=v}$ with decidable rational subset, or submonoid, membership problem.

\smallskip

Of course this problem may well be difficult to answer given the fact that the word problem for one-relation monoids remains open but, motivated by the connection with the reduction results by Adian \& Oganesian and Guba mentioned above, there is still strong motivation for developing a better understanding of the submonoid and rational subset membership problems in one-relation monoids.  

The general study of one-relation monoids $\Mpres{A}{u=v}$ typically splits into cases that, roughly speaking, give a measure of how far away the monoid is from being a group. In more detail, the study of one-relation monoids naturally divides into the investigation of so-called special, and more generally subspecial, monoids on the one hand, and those that are not subspecial, on the other. As we will explain in more detail below, associated with any one-relation subspecial monoid is a unique positive one-relator group that arises as a maximal subgroup of the monoid and the algorithmic properties of the monoid (e.g. the word problem) are typically controlled by properties of this positive one-relator group. 

All of the remaining one-relation monoids, i.e. those that are non-subspecial, are far away from being groups in the sense that the only idempotent such a monoid contains is its identity, and the group of units of the monoid is trivial. Recall from above that the word problem for one-relation monoids has been reduced to the problem of solving the word problem for one-relation monoids of the form $\Mpres{a,b}{bUa=aVa}$ and of the form $\Mpres{a,b}{bUa=a}$. All of the monoids in these two families are non-subspecial and it is natural to ask whether the rational subset membership (or submonoid) problems are decidable in the non-subspecial case, and in particular for monoids in these two classes. It is not difficult to show that the word problem in such monoids, which are left cancellative, reduces to the left divisibility problem, i.e. the problem of deciding membership in the principal right ideals $wA^\ast$.

Any principal ideal is clearly a rational subset so a better understanding of the membership problem in rational subsets of monoids of this form is important for the study of the word problem. We will see below how to construct monoids of both forms $\Mpres{a,b}{bUa=a}$ and $\Mpres{a,b}{bUa=aVa}$ with undecidable rational subset membership problem. Constructing such examples is more difficult than in the subspecial case since the monoids do not embed any subgroups (apart from the trivial group) so the usual approach of embedding the RAAG $A(P_4)$ is not available to us in those cases. This will be discussed in more detail below. 
\subsection{Special and subspecial monoids}

As we have already seen above, the one-relation monoids that are groups all admit presentation of the form $\Mpres{A}{w=1}$. These are usually called \emph{special one-relation monoids} in the literature. It follows from results of Adian \cite{Adian1966} that the group of units of any special one-relation monoid is a positive one-relator group, and there is an algorithm (called Adian's \textit{overlap algorithm}) that computes a presentation for the group of units of the monoid. In more detail, the algorithm computes a factorization $w \equiv u_1 \ldots u_k$ into non-empty words $u_i$ that all represent invertible elements of the monoid, and no proper non-empty prefix of $u_i$ represents an invertible element. The set of factors $\{u_i : 1 \leq i \leq k\}$ is \textit{overlap-free}, in the sense that no non-empty prefix (resp. suffix) of a word in this set is equal to a non-empty suffix (resp. prefix) of another word in this set.  This is called the decomposition of the relator into minimal invertible pieces, and the group of units of the monoid is then isomorphic to $\Mpres{B}{b_{u_1} \ldots b_{u_k}}$ where $B = \{ b_{u_i} : 1 \leq i \leq k \}$.  So the group of units $\Mpres{B}{b_{u_1} \ldots b_{u_k}}$ of $M=\Mpres{A}{w=1}$ is a positive one-relator group, and if that group has undecidable submonoid, or rational subset, membership problem then so does $M$.  This gives the most straightforward way of using our results to construct non-group one-relation monoids with undecidable submonoid membership problem.   
\begin{prop}\label{prop_specialmonoidexamples} 
Let $m, n \geq 2$ and let $A$ be a finite alphabet and let $\alpha, \beta \in A^+$ such that $\{\alpha,\beta\}$ is overlap-free.  
Then the group of units of 
\[
N = \Mpres{A}{(\beta \alpha^n)^m(\alpha^n \beta)^m = 1}
\]
is isomorphic to 
$M_{m,n}$
and hence 
$N$  contains a fixed finitely generated submonoid in which membership is undecidable. The monoid $N$ is a group if and only if $|\alpha| = |\beta| = 1$.  
\end{prop}
\begin{proof} 
The fact that the group of units of $N$ is $M_{m,n}$ follows by using the fact that $\{\alpha,\beta\}$ are overlap-free and applying Adian's overlap algorithm \cite{Adian1966}. If $|\alpha| \geq 2$ then the first letter of $\alpha$ is right invertible but not left invertible in $N$ from which it follows that $N$ is not a group; similarly if $|\beta| \geq 2$.      
\end{proof}
For example taking $\alpha = xy$ and $\beta = xxyy$ in 
Proposition~\ref{prop_specialmonoidexamples}
gives that the following monoid, which is not a group since $x$ is not invertible, has undecidable submonoid membership problem  
\[
\Mpres{x, y}{(xxyy (xy)^n)^m((xy)^n xxyy)^m = 1}
\]
for $m,n \geq 2$.

More generally, a one-relation monoid $\Mpres{A}{u=v}$ with $|v| \leq |u|$ is called \textit{subspecial} if $u \in vA^* \cap A^*v$. The following proposition follows from known results, and explains the close connection between these one-relation monoids and the class of positive one-relator groups. 

\begin{prop}\label{prop_subspecial}  
Let $M$ be a subspecial one-relation monoid, i.e. let $M=\Mpres{A}{u=v}$ where $|v| \leq |u|$ and $u \in vA^* \cap A^*v$. 
\begin{enumerate} 
\item[(i)] If $v=1$ is the empty word then $M$ is a special monoid with group of units $H$ isomorphic to a positive one-relator group. 
\item[(ii)] If $v \neq 1$ then the group of units is trivial, the monoid contains non-trivial idempotents, and there is a fixed positive one-relator group $H$ such that the maximal subgroup (i.e. group $\mathscr{H}$-class) associated to any non-trivial idempotent is isomorphic to $H$. 
  \end{enumerate}
In both cases, there is an algorithm that computes a presentation and a generating set for the positive one-relator group $H$ from the given presentation of $M$. In particular, if $M$ has decidable submonoid (resp. rational subset) membership problem then so does the positive one-relator group $H$.   
\end{prop}

There is evidence in the literature that the converse of this proposition should be true in the case of rational subset membership, that is, in both cases of Proposition~\ref{prop_subspecial} we expect that the monoid $M$ will have decidable rational subset membership problem if and only if the positive one-relator group $H$ does. For instance if the group $H$ is a virtually free group (and hence has decidable rational subset membership problem) then it was proved in \cite{NybergBrodda2022a, nybergbroddacompression} that $M$ has decidable rational subset membership problem. In this sense we expect the problem of classifying subspecial monoids with decidable rational subset membership problem to be equivalent to the problem of classifying the positive one-relator groups with this property. 

Given that we now have examples of positive one-relator groups with undecidable submonoid membership problem, Proposition~\ref{prop_subspecial} gives a recipe for constructing many more subspecial examples, by realizing our examples as the group $H$ in the proposition.  Rather than developing that theory in full here, we content ourselves by giving an example from which it should be clear that many other examples could be constructed using the same approach.

\begin{example}\label{ex_compressSpecial}
For $m,n \geq 1$ define $T_{m,n}$ to be the monoid  
\[
\Mpres{z,t}{((ztz^2t)^2(ztz^3t)^2 (ztz^2tztz^3t)^n)^m((ztz^2tztz^3t)^n (ztz^2t)^2(ztz^3t)^2 )^m = zt}. 
\]
If we compress this monoid (in the sense of \cite{Lallement1974, Kobayashi2000}) with respect to $zt$ we obtain the monoid  
\[
\Mpres{x, y}{(xxyy (xy)^n)^m((xy)^n xxyy)^m = 1}, 
\]
and from above the group of units of this monoid is isomorphic to 
\[
M_{m,n} = \Mpres{a,b}{(b a^n)^m(a^n b)^m = 1}. 
\]
It then follows from \cite{Lallement1974} that $T_{m,n}$ contains non-identity idempotents and the maximal subgroup of any of these idempotents is isomorphic to the positive one-relator group $M_{m,n}$. It follows that for all $m, n \geq 2$ 
the monoid $T_{m,n}$ embeds the group $A(P_4)$ and hence contains a fixed finite generated submonoid in which membership is undecidable.   
\end{example}

\subsection{Non-subspecial monoids} 

As explained in the beginning of this section, all of the remaining one-relation monoids, i.e. those that are non-subspecial, are far away from being groups, containing no non-trivial subgroups. These non-subspecial monoids will be the topic of the remainder of this section, and also the next section where we develop new methods for constructing examples of these forms with undecidable rational subset membership problem. We now turn our attention to the class of \textit{monadic} one-relation monoids $\Mpres{a,b}{bUa=a}$ discussed in the introduction to this paper. As explained there, one major motivation for the work done in this paper is the work of Guba \cite{Guba1997} which reduces the word problem in these monoids to the membership problem in certain submonoids of particular positive one-relator groups. In the next section we will give an infinite family of monadic one-relation monoids $\Mpres{a,b}{bUa=a}$, each of which has undecidable rational subset membership problem. 

In his paper \cite{Guba1997}, Guba associates with any $\Mpres{a,b}{bUa=a}$ a positive one-relator group $G$ such that the word problem in $\Mpres{a,b}{bUa=a}$ reduces to solving the membership problem within a certain submonoid of $G$.
Of course if $G$ has decidable submonoid membership problem then $\Mpres{a,b}{bUa=a}$ will have decidable word problem, so the only cases of interest now are those where the positive one relator group $G$ does \emph{not} have decidable submonoid membership problem. We now know from the results above that such positive one-relator groups $G$ do exist, so the next step is to seek monoids $\Mpres{a,b}{bUa=a}$ such that the associated positive one-relator groups arising from Guba's theory are isomorphic to the positive groups $G_{m,n}$ that we defined and investigated above. We will now identify one such class of monadic one-relation monoids, and give some of their basic properties, before going on to study rational subset membership in this class in depth in the next section (\S\ref{Sec:Guba-type-monoids}).

Let $G_{m,n} = \Gpres{x,y}{x^my^nx^my^{-n}}$. By Remark \ref{A(P_4) inside G_{m,n}} the subgroup 
\[
H = \Ggen{yx^my^{-1}, y^{2n}, x^m,xy^{2n}x^{-1}}
\]
is isomorphic to $A(P_4)$ in $G_{m,n}$.

The substitution $y \mapsto a, x \mapsto ba^n$ defines an isomorphism between $G_{m,n}$ and the group:
\[
M_{m,n} = \Mpres{a, b}{(b a^n)^m(a^n b)^m = 1}.
\]
We know from the results in \S\ref{Sec:SMMP-in-positive-OR-groups} that $M_{m,n}$ is a one-relator group defined by a positive word and with an undecidable submonoid membership problem. 

In the new presentation, the subgroup $H$ is given as: 
\[
H = \Ggen{a(ba^n)^ma^{-1}, a^{2n}, (ba^n)^m, ba^{2n}b^{-1}}.
\]

Let $W_{m,n} \equiv (b a^n)^m(a^n b)^m$, and let $Q_{m,n}$ be the longest proper suffix of $W_{m,n}$, i.e. $W_{m,n} \equiv b Q_{m,n}$. Consider the monoid 
\[
R_{m,n} = \text{Mon}\langle \, a,\, b \mid a = bQ_{m,n}a \,\rangle.
\]
Applying Oganesian's algorithm (see \cite{Guba1997}) to the monoid 
$R_{m,n}$, we see that the suffix monoid of $R_{m,n}$ embeds in the one-relator group 
\[
G_{m,n} = \text{Gp}\langle \, a,\, b \mid (b a^n)^m(a^n b)^m = 1\,\rangle = M_{m,n}.
\]
Indeed, the suffixes of $Q_{m,n}a$ that start with $a$ are left-divisible by $a$ (which is the first suffix), while the suffixes of $Q_{m,n}a$ that start with $b$, start with $ba$ as well; so, these suffixes are left-divisible by $ba$ (which is the second suffix). Moreover, as $a = bQ_{m,n}a \equiv (ba^n)^m(a^nb)^m$, $a$ is left-divisible by $ba$. By Oganesian's algorithm, this suffices to show that the suffix monoid of $R_{m,n}$ embeds in $G_{m,n}$. 

In the next section, we will prove the undecidability of the rational subset membership problem in the monoid $R_{m,n}$. Before doing so, we will solve the word problem in $R_{m,n}$. We will do so by means of a finite complete rewriting system. First, note that the word $bQ_{m,n}a$ has self-overlaps; in fact, all the words of the form~$(ba^n)^i ba$ (for~$0 \leq i \leq m-1$) are both a prefix and a suffix. Thus while the rewriting system $bQ_{m,n}a \rightarrow a$ defines $R_{m,n}$, it is not complete. It can, however, be completed to one:

\begin{lemma}\label{lem: FCRS}
The monoid $R_{m,n}$ admits a finite complete rewriting system $S$ on the alphabet $\{a, b\}$, and with the following rules:
\begin{itemize}[topsep=0pt]
\item[(i)] $(ba^n)^m (a^nb)^{m} a \longrightarrow a$,
\item[(ii)] $(ba^n)^m (a^nb)^{m-i} a^n a \longrightarrow (a^nb)^{m-i} a^n (a^n b)^{m} a \: (1 \leq i \leq m)$.
\end{itemize}
\end{lemma}

We will denote rule (i) as $\alpha_0 \to \beta_0$, and (ii) as $\alpha_i \to \beta_i$ ($1 \leq i \leq m$), respectively. To show that the system $S$ in Lemma~\ref{lem: FCRS} is complete, it is enough to show that it is Noetherian and locally confluent by Newman's Lemma (see e.g. \cite[Lemma 12.15]{holt2005handbook}). Obviously, it is Noetherian, as can be seen by using the shortlex order. We must therefore only show that our system is locally confluent. For this, we use the following lemma:

\begin{lemma}[Lemma 12.17 in \cite{holt2005handbook}]\label{lem: locally confluent}
The system $S$ is locally confluent if and only if for all pairs of rules $(l_1, r_1), (l_2, r_2) \in S$, the following
conditions are satisfied:
\begin{itemize}[topsep=0pt, partopsep=0pt]
\item[(i)] If $l_1 = us$ and $l_2 = sv$ with $u, s, v \in \{a,b\}^*$ and $s \neq \varepsilon$, then there exists $w \in \{a,b\}^*$ with $r_1 v \rightarrow ^* w$, and $u r_2 \rightarrow ^* w$.
\item[(ii)] If $l_1 = usv$ and $l_2 = s$ with $u, s, v \in \{a,b\}^*$ and $s \neq \varepsilon$, then there exists $w \in \{a,b\}^*$ with $r_1 \rightarrow^* w$, and $ur_2v \rightarrow^* w$.
\end{itemize}
\end{lemma}

\begin{proof}[Proof of Lemma \ref{lem: FCRS}]
Note that the second condition of Lemma \ref{lem: locally confluent} does not apply to our system, as none of the $\alpha_i$ is a subword of another $\alpha_j$.

Note also that the first condition of Lemma \ref{lem: locally confluent} can only be applied for $l_1 = \alpha_0$ and $l_2 = \alpha_i$ for $0 \leq i \leq m$, because other pairings do not overlap. 

For $l_1 = \alpha_0$ and $l_2 = \alpha_i$ set $s_j = (ba^n)^j ba$ for $0 \leq j \leq m-1$. We want $\alpha_0 = u_js_j$ and $\alpha_i = s_j v_{i,j}$, so: $u_j = (ba^n)^m (a^nb)^{m-1-j}a^n$; $v_{0,j} = a^{n-1}(ba^n)^{m-1-j}(a^nb)^{m}a$; while for all other $i > 0$ one has $v_{i,j} = a^{n-1}(ba^n)^{m-1-j}(a^nb)^{m-i}a^n a$.

Now we want to find a word $w_{i,j}$ with $\beta_0 v_{i,j} \rightarrow ^* w_{i,j}$, and $u_j \beta_i \rightarrow ^* w_{i,j}$. Recall that $\beta_0 = a$, and $\beta_i = (a^nb)^{m-i} a^n (a^n b)^{m} a$. Actually, one can see  that $w_{i,j} = \beta_0 v_{i,j} = a v_{i,j}$, i.e. $w_{i,j} = \beta_0 v_{i,j}$ is reduced with respect to our system, and $u_j \beta_i$ gets reduced to $w_{i,j}$.
\end{proof}

\section{Undecidability in monadic one-relation monoids}\label{Sec:Guba-type-monoids}

The goal of this section is to prove the following result, the proof of which will use the results of the previous sections. 

\begin{theorem}\label{thm_UndecRatSubsetGubaType}
For all $m,n \geq 2$,  
the monoid 
\[
\M=
\Mpres{a, b}{(b a^n)^m(a^n b)^ma = a}
\]
contains a fixed rational subset in which membership is undecidable.
\end{theorem}

Since the first letters of the two sides of the defining relation are distinct, the monoid $\M$ is left-cancellative by Adian \cite[Theorem~3]{Adian1960}, so the only idempotent in $\M$ is the identity. The group of units in $\M$ is trivial since the defining relation is not of the form $w=1$. Hence, $\M$ does not contain any non-trivial groups; in particular, $\M$ does not embed $A(P_4)$. So Theorem~\ref{thm_UndecRatSubsetGubaType} cannot be proved by embedding $A(P_4)$, or embedding any other group for that matter. Theorem~\ref{thm_UndecRatSubsetGubaType} gives the first known examples of non-subspecial one-relation monoids with undecidable rational subset membership problem.  

To prove the theorem we will need to introduce a new approach which involves embedding the trace monoid $T(P_4)$ into $\M$ in a certain way. It is not the case that every left-cancellative monoid that embeds $T(P_4)$ has undecidable rational subset membership problem. For example, $T(P_4)$ itself is left-cancellative (it is even group-embeddable) and has decidable rational subset membership problem \cite{Diekert1990}. So the way in which the trace monoid embeds into $\M$ will be vital.       

Recall \cite[Chapter~II]{Howie1995} that two elements $x$, $y$ in a monoid $T$ are said to be $\gl$-related if $Tx = Ty$. Also note that it is immediate from the definition that $\gl$ is a right congruence, i.e. if $x \gl y$ then $xz \gl yz$ for any $z \in T$. This will be used implicitly throughout our proofs below. We will prove the following general result and then show that the hypotheses are satisfied by the one-relation monoid $\M$ above.    

\begin{theorem}\label{thm_LClassEmbeddingTraceMonoidImproved} 
Let $M$ be a finitely generated left-cancellative monoid and let $U \subseteq M$ such that $u v \gl v$ for all $u, v \in U$. 
If $\Mgen{U}$ is isomorphic to the trace monoid $T(P_4)$ then   
$M$ contains a fixed rational subset in which membership is undecidable.  
\end{theorem}

By the trace semigroup of $P_4$ we mean the semigroup defined by the semigroup presentation $\Spres{a,b,c,d}{ab=ba, bc=cb, cd=dc}$. 

\begin{cor} 
If a left-cancellative monoid embeds a copy of the trace semigroup of $P_4$ 
that is 
contained in a single $\gl$-class of the monoid, then the monoid contains a fixed rational subset in which membership is undecidable.  
\end{cor}

In addition to applying to our one-relation monoid example, the general result Theorem~\ref{thm_LClassEmbeddingTraceMonoidImproved} also has the following application to groups, since in a group every pair of elements are clearly $\gl$-related.  

\begin{cor}\label{cor_T4EmbedsInGroup}
If $G$ is a finitely generated group which embeds the trace monoid $T(P_4)$ then $G$ contains a fixed rational subset in which rational subset membership is undecidable. 
\end{cor}

Given this, a natural question is whether a group can embed $T(P_4)$ but not $A(P_4)$; this will be discussed at the end of this section.   

Before proving Theorem~\ref{thm_LClassEmbeddingTraceMonoidImproved} we will show how it can be applied to prove Theorem~\ref{thm_UndecRatSubsetGubaType}. For that we need the following lemma about the $\gl$-relation in one-relation monoids of the same form as our example.   

\begin{lemma}\label{lem_LClassofa} 
Let  
\[
M = \Mpres{a,b}{ba^{i_1}ba^{i_2} \ldots ba^{i_k} b a =a}
\] 
with $i_j \geq 1$ for all $1 \leq j \leq k
$. 
Then $w \gl a$ in $M$ for every word $w \in \{a,ba\}^+$. 
  \end{lemma}
\begin{proof} 
Set $r \equiv ba^{i_1}ba^{i_2} \ldots ba^{i_k} b a$. Since $r$ begins and ends with $ba$ we can write $r \equiv \alpha ba \equiv ba \gamma$.  As $a=r=ba\gamma$ in $M$ and $\gamma$ ends in the letter $a$, it follows that $a \gamma \gl a$. In $M$ we have $ a \gamma = \alpha ba \gamma = \alpha a $ hence $\alpha a \gl a$.  Since $r \equiv \alpha ba$ and $i_k \geq 1$ it follows that  the last letter of $\alpha$ equals $a$, so we can write $\alpha \equiv \alpha' a$. But now $\alpha a \gl a$ implies $\alpha' a^2 \gl a$ so $\delta \alpha' a^2 = a$ for some word $\delta$. It follows that $a^2 \gl a$ in $M$.              

Next observe that $ba \gl a$ since $ba$ is a suffix of $r$, from which it follows that $ba(ba) \gl a(ba)$ and also $ba(a) \gl a(a)$.  Also $aba \gl a$ since $aba$ is a suffix of $r$.   We have shown $aa \gl a$, $a(ba) \gl a$, $(ba)(ba) \gl aba \gl a$ and $(ba)a \gl aa \gl a$, that is, all two-factor products of the words $\{a,ba\}$ are $\gl$-related to $a$.   Now the result follows by induction since setting $w_1 \equiv a$, $w_2 \equiv ba$, for any product  $ w_{i_1} w_{i_2} \ldots w_{i_k} $ with $k \geq 3$  from above we have $w_{i_1} w_{i_2} \gl a$ 
from which, since $\gl$ is a right congruence, it follows that   
$
w_{i_1} w_{i_2} w_{i_3} \ldots w_{i_k} \gl a w_{i_3} \ldots w_{i_k}
$
where by induction $a w_{i_3} \ldots w_{i_k} \equiv w_1 w_{i_3} \ldots w_{i_k} \gl a$. 
\end{proof}

The key to applying Theorem~\ref{thm_LClassEmbeddingTraceMonoidImproved} to $\M$ is to find an appropriately embedded copy of the trace monoid $T(P_4)$. This is achieved in the following lemma.  

\begin{lemma}\label{lemma: copy of A(P4)}
For all $m, n \geq 2$ the submonoid $T_4 = \Mgen{X}$ of $\M$ generated by
\[
X =  
\{ 
a(ba^n)^{m-1} b a^{n-1}, \;
a^{2n},  \;
(ba^n)^m,  \;
ba^n a^{2n} (ba^n)^{m-1}
\} 
\]
is isomorphic to the trace monoid $T(P_4)$.
\end{lemma}
\begin{proof}
Set 
\[
\alpha
=a(ba^n)^{m-1} b a^{n-1}, \;
\beta
=a^{2n},  \;
\gamma
=(ba^n)^m,  \;
\mu
=ba^n a^{2n} (ba^n)^{m-1}.
\]
Let $H = \Ggen{a(ba^n)^ma^{-1}, a^{2n}, (ba^n)^m, ba^{2n}b^{-1}} \leqslant M_{m,n}$, 
and denote by $A, B, C, D$ its $4$ generators, in the given order,  
recalling that $M_{m,n} = \Gpres{a, b}{(b a^n)^m(a^n b)^ma = a}$. In Section~\ref{Sec:SMMP-in-positive-OR-groups}, Remark \ref{A(P_4) inside G_{m,n}} was applied to show that $H \cong A(P_4)$ with $A$, $B$, $C$ and $D$ corresponding to the vertices in the path $P_4$ in this order. It then follows from a sequence of easy Tietze transformations that $\Ggen{\alpha,\beta,\gamma,\mu} \leq M_{m,n}$ is also isomorphic to $A(P_4)$ 
since $\alpha = a(ba^n)^{m-1} b a^{n-1}$ is equal to $A$ in $M_{m,n}$,
$\beta = B$, $\gamma = C$,   
and    
\[
\mu = ba^n a^{2n} (ba^n)^{m-1} = ba^{2n}a^n (ba^n)^{m-1} = ba^{2n}b^{-1}(ba^n)^m  = DC
\]
is equal to $DC$ in $M_{m,n}$. Since any trace monoid naturally embeds in its corresponding right-angled Artin group, it follows that 
$\Mgen{\alpha,\beta,\gamma,\mu} \leq M_{m,n}$ is isomorphic to the trace monoid $T(P_4)$ where $\alpha$, $\beta$, $\gamma$, $\mu$ correspond to the vertices on the path $P_4$ in this order.  
Let $\phi$ be the homomorphism $\phi: R_{m,n} \rightarrow M_{m,n}$ induced by the identity map on $\{a,b\}$. 
Then 
\[
\phi(T_4) = \phi(\Mgen{X}) = \Mgen{\phi(X)} =  \Mgen{\alpha,\beta,\gamma,\mu} \leq M_{m,n}
\]
is isomorphic to the trace monoid $T(P_4)$ where $T_4 = \Mgen{X} = \Mgen{\alpha,\beta,\gamma,\mu} \leq R_{m,n}$. So $\phi$ induces a surjective homomorphism from $T_4$ onto $\phi(T_4) \leq M_{m,n}$. To show that this defines an isomorphism between $T_4$ and $\phi(T_4)$ we need to show that $\phi$ is injective on the set $T_4$ which, since $\phi(T_4)$ is isomorphic to $T(P_4)$, means we need to show that the defining relations of the trace monoid hold between the generators of $T_4$ in the monoid $R_{m,n}$. Hence to complete the proof we just need to show that $\alpha \beta = \beta \alpha$, $\beta \gamma = \gamma \beta$ and $\gamma \mu = \mu \gamma$ all hold in $R_{m,n}$.

Using the rewriting rule (ii) for $i = m$ from Lemma \ref{lem: FCRS} we obtain
\begin{equation}\label{equation:ruleTwoi=m}
(ba^n)^m a^n a = a^n (a^n b)^{m} a.
\end{equation}
Multiplying Equation \eqref{equation:ruleTwoi=m} by $a^{n-1}$ on the right, we obtain $(ba^n)^m a^{2n} = a^{2n} (ba^n)^{m}$; i.e. $\gamma\beta = \beta\gamma$ in $R_{m,n}$.
Multiplying Equation \eqref{equation:ruleTwoi=m} by $a$ on the left, and $a^{n-2}$ on the right (note that $n \geq 2$) and writing $m = (m-1) + 1$ we obtain:
\[
[a(ba^n)^{m-1} ba^{n-1}]a^{2n}= a^{2n} [a(ba^n)^{m} ba^{n-1}],
\]
which means that $\alpha\beta = \beta\alpha$ in $R_{m,n}$.
Lastly, 
note that $\gamma = (ba^n)^m$ commutes with all the  words $ba^{n}$,  $a^{2n}$, and $(ba^n)^{m-1}$. So we obtain:
\[
\mu \gamma = [ba^n a^{2n} (ba^n)^{m-1}] \gamma = \gamma [ba^n a^{2n} (ba^n)^{m-1}] = \gamma \mu,
\]
as required. 
\end{proof}

We are now in a position to prove the main result of this section, which we presented at the very beginning.

\begin{proof}[Proof of Theorem~\ref{thm_UndecRatSubsetGubaType}]
Let $m,n \geq 2$,  and set $\M= \Mpres{a, b}{(b a^n)^m(a^n b)^ma = a}$. 
By Lemma~\ref{lemma: copy of A(P4)} the submonoid $T_4 = \Mgen{X}$ of $\M$ generated by
\[
X =  
\{ 
a(ba^n)^{m-1} b a^{n-1}, \;
a^{2n},  \;
(ba^n)^m,  \;
ba^n a^{2n} (ba^n)^{m-1}
\} 
\]
is isomorphic to the trace monoid $T(P_4)$, and since every word in $X$ belongs to  $\{a,ba\}^+$ it follows from 
Lemma~\ref{lem_LClassofa} that all non-empty products of elements of $X$ are $\gl$-related to each other, since they are all $\gl$-related to $a$.   
Since $M$ is left cancellative by   
Adian \cite[Theorem~3]{Adian1960}, 
the result now follows by applying 
Theorem~\ref{thm_LClassEmbeddingTraceMonoidImproved}. 
\end{proof}

The remainder of this section will be devoted to proving Theorem~\ref{thm_LClassEmbeddingTraceMonoidImproved} and then discussing some of its other consequences.  

First, the following result identifies a sufficient condition for constructing rational subsets of left-cancellative monoids in which membership is undecidable.   

\begin{theorem}\label{thm: general theorem undecidability}
Let $M$
be a left-cancellative monoid with a finite generating set $A$. 
Suppose that there exist rational languages $L, K, \overline{L} \subseteq A^*$,   
a surjective map $\alpha: L \longrightarrow \overline{L}$ with $\alpha(l) = \overline{l}$, and a fixed word $w \in A^*$, satisfying the following properties:
\begin{itemize}
\item[(i)] $\overline{l} l w = w$ in $M$, for all $l \in L$, and
\item[(ii)] it is undecidable whether there exists a pair $(l,k) \in L \times K$ satisfying $lw^i = k$ in $M$, for a given $i \in \N$.
\end{itemize}
Then the rational language $\overline{L}K \subseteq A^*$ defines a fixed rational subset of $M$ in which membership is undecidable.
\end{theorem}
\begin{proof}
The language $\overline{L}K \subseteq A^*$ is clearly rational as both $\overline{L}$ and $K$ are. 

Let $i \in \mathbb{N}$. 
Since $M$ is left cancellative it follows
that for any  $(l,k) \in L \times K$ we have 
$lw^i = k$ in $M$ if and only if $\overline{l}lw^i = \overline{l}k$. 
Applying condition (i) gives $\overline{l}lw^i = \overline{l}lww^{i-1}  = w^i$, hence    
$lw^i = k$ in $M$ if and only if $w^i = \overline{l}k$.   

It follows that 
for a given $i \in \mathbb{N}$ 
there exists $(l,k) \in L \times K$ satisfying 
$lw^i = k$ if and only if there exists $(l,k) \in L \times K$ satisfying  
$w^i = \overline{l}k$ which is true if and only if $w^i \in \overline{L} K$,
since $\alpha$ being surjective implies that $\overline{L} = \{\overline{l} : l \in L \}$.    
But the former problem is undecidable by assumption (ii) and hence the latter problem must also be undecidable, i.e. membership in the rational subset $\overline{L}K$ of $M$ is undecidable. 
\end{proof}

Condition (ii) in the above theorem comes from the following result of Lohrey and Steinberg \cite{Lohrey2008} about the trace monoid $T(P_4)$.   

\begin{lemma}[Corollary of Proof of Theorem~2 in \cite{Lohrey2008}] \label{thm: nondecidability}
Let $T$ be the trace monoid of $P_4$ defined by    
\[T = \Mpres{u_1, u_2, u_3, u_4}{
u_1u_2 = u_2u_1, 
u_2u_3 = u_3u_2, 
u_3u_4 = u_4u_3}.
\]
Then there are two fixed rational subsets $Q, R \subseteq \{u_1, u_2, u_3, u_4\}^*$, 
with $Q$ not containing the empty word,    
such that 
it is undecidable whether there exists a pair $(x,y) \in Q \times R$ satisfying $x(u_2)^i = y$ in $T$, for a given $i \in \N$.
\end{lemma}
\begin{proof} 
In the paper \cite[Proof of Theorem~2]{Lohrey2008} the authors take $\Sigma = \{a,b,c,d\}$ and work in the trace monoid $S = \Mpres{\Sigma}{ab=ba, bc=cb, cd=dc}$. They show that there is a fixed rational language $L \subseteq \Sigma^*$ and a family of languages $K_{m,n}$ with $m, n \in \mathbb{N}$ such that it is undecidable whether $K_{m,n} \cap L \neq \varnothing$ for given $m, n \in \mathbb{N}$. For our purposes the important thing is that $K_{m,n} = b^{2^m3^n} \Omega$ where $\Omega \subseteq \Sigma^*$ is a certain fixed rational language with the property that $\Omega$ does not contain the empty word.\footnote{In fact in \cite[Proof of Theorem~2]{Lohrey2008}, the authors use $K_{m,n}= b^{2^m3^n} \Omega$ where $\Omega = a(d(cb)^+a)^*dc^*$. However, this definition appears to have a typo, since for the equality in the displayed equation in the fourth-from-last line of their proof to be true, one should really take $\Omega$ to be $a(d(cb)^+a)^*dc^+$ because of the condition $j_l \geq 1$. This small change does not have any impact on the conclusion of their result or on our application of it.}        
Their argument shows that it is undecidable whether $b^{2^m3^n} \Omega \cap L \neq \varnothing$ for given $m, n \in \mathbb{N}$. 
It follows from the defining relations in the trace monoid $S$ that for any two words $u, v \in \Sigma^*$ we have $u=v$ in $S$ if and only if $\mathrm{rev}(u) = \mathrm{rev}(v)$ in $S$ where $\mathrm{rev}$ denotes the reverse of a word.  
Hence if we set $\Omega' = \{ \mathrm{rev}(u) : u \in \Omega \}$ and $L' = \{ \mathrm{rev}(l) : l \in L \}$ then $\Omega'$ and $L'$ are rational subsets of $\Sigma^*$, since word reversal clearly preserves the property of being a regular language. Furthermore, $\Omega'$ and $L'$ have the property that it is undecidable whether $\Omega' b^{2^m3^n} \cap L' \neq \varnothing$ for given $m, n \in \mathbb{N}$. 
In particular this means that 
there are fixed rational subsets $\Omega', L' \subseteq \Sigma^*$ such that 
it is undecidable whether $\Omega' b^{i} \cap L' \neq \varnothing$ for given $i \in \mathbb{N}$. 
Finally if we translate this into the notation of the trace monoid as defined in the statement of the lemma by substituting $a \mapsto u_1$, $b \mapsto u_2$, $c \mapsto u_3$ and $d \mapsto u_4$ we obtain the result, where $Q$ does not contain the empty word since $\Omega$ does not.  
  \end{proof}

We will also need the following lemma about a certain mapping on words that preserves the property of being a regular language. 

\begin{lemma}\label{lem_PresRegular}
Let $A = \{a_1, \ldots, a_n \}$ and for every pair $(i,j)$ of natural numbers with 
$i,j \in \{1, \ldots, n \}$  
choose and fix some $b_{i,j} \in A^*$. 
For every word $w = a_{i_1} a_{i_2} \ldots a_{i_{k-1}} a_{i_k}\in A^*$ of length at least $2$ define  
\[
\phi(a_{i_1} a_{i_2} \ldots a_{i_{k-1}} a_{i_k}) = 
b_{i_{k-1},i_k} 
b_{i_{k-2},i_{k-1}} 
\ldots 
b_{i_2,i_3} 
b_{i_1,i_2}.  
\]
If $L \subseteq A^*$ is a regular language, and every word in $L$ has length at least two, then $\phi(L) \subseteq A^*$ is also a regular language.     
  \end{lemma}
\begin{proof} 
Let $L \subseteq A^*$ be a regular language, and suppose that every word in $L$ has length at least two. 
Let $\sigma:A^* \rightarrow A^*$ be the homomorphism defined by $a_i \mapsto a_i^2$ and set $L_1 = \sigma(L)$. Since the class of regular languages is closed under taking homomorphisms, it follows that $L_1$ is a regular language.     
Since every word in $L$ has length at least two it follows that every word in $L_1$ has length at least four. 
Now let $L_2$ be the language obtained by taking each word from $L_1$ and deleting the first letter and the last letter. Note that every word in $L_2$ has even length, and has length at least two.   
Since the left or right quotient of a regular language is again regular, and deletion of the last resp. first letter is an example of taking a left resp. right quotient of a language by a regular language, it follows that $L_2$ is a regular language.  
Let $g:L_1 \rightarrow L_2$ be the map that deletes the first and last letter of the input word, and $L_3$ the reversal of the language $L_2$, i.e. the language obtained by reversing every word.

Now define a new alphabet $C = \{ c_{i,j} : 1 \leq i,j \leq n \}$ and a homomorphism $\theta:C^* \rightarrow A^*$ defined by $c_{i,j} \mapsto a_i a_j$.     
Note that the homomorphism $\theta$ is clearly injective with image the set of all word in $A^*$ of even length.   
Since regularity is preserved under taking inverse images of homomorphisms, it follows that for any regular language $W \subseteq A^*$ the language $\theta^{-1}(W) \subseteq C^*$ is regular.   
Let $\rho:C^* \rightarrow C^*$ be the word reversing map which also preserves regularity. 

Now, given any word 
$w = a_{i_1} a_{i_2} \ldots a_{i_{k-1}} a_{i_k}\in A^*$ of length at least $2$ we have 
\[
g(\sigma(w)) = 
a_{i_1} 
a_{i_2} 
a_{i_2} 
a_{i_3} 
a_{i_3} 
\ldots
a_{i_{k-2}} 
a_{i_{k-2}} 
a_{i_{k-1}} 
a_{i_{k-1}} 
a_{i_k}
\]
is a non-empty word of even length and then 
\[
(\rho \circ \theta^{-1} \circ g \circ \sigma)(w) = 
c_{i_{k-1},i_k} 
c_{i_{k-2},i_{k-1}} 
\ldots 
c_{i_2,i_3} 
c_{i_1,i_2} 
\] 
Finally let $\gamma:C^* \rightarrow A^*$ be the homomorphism defined by $c_{i,j} \mapsto b_{i,j}$, so that  
\[
(\gamma \circ \rho \circ \theta^{-1} \circ g \circ \sigma) (a_{i_1} a_{i_2} \ldots a_{i_{k-1}} a_{i_k}) = 
b_{i_{k-1},i_k} 
b_{i_{k-2},i_{k-1}} 
\ldots 
b_{i_2,i_3} 
b_{i_1,i_2}.  
\]
It follows that the mapping $\phi$ in the statement of the lemma satisfies  
$\phi = \gamma \circ \rho \circ \theta^{-1} \circ g \circ \sigma$ and  
since as explained above all of the maps in this composition preserve regularity of follows that $\phi(L)$ is a regular language.    
  \end{proof}

We now have everything we need to prove our general theorem. 

\begin{proof}[Proof of Theorem~\ref{thm_LClassEmbeddingTraceMonoidImproved}] 
Let $M$ be a finitely generated left-cancellative monoid and let $U \subseteq M$ such that $u v \gl v$ for all $u, v \in U$, and $\Mgen{U}$ is isomorphic to the trace monoid $T(P_4)$.
Since every generating set of $T(P_4)$ must contain the four standard generators of the monoid that correspond to the four vertices of the path $P_4$ (this follows from the fact that the relations in the presentation of $T(P_4)$ are length preserving so one cannot recover the letters as products of words of greater length) it follows that $U$ contains a subset $Y = \{u_1, u_2, u_3, u_4\}$ corresponding to the standard generators of the trace monoid. 
Since the rational subsets of a monoid do not depend on the choice of finite generating sets, without loss of generality we can take the generating set $A$ for $M$ in the statement of the theorem to be a finite generating set for $M$ 
so that it contains the subset $Y \subseteq A$. 
So $Y \subseteq A$ with $Y = \{u_1, u_2, u_3, u_4\}$ where $T = \Mgen{Y}$ is isomorphic to the trace monoid $T(P_4)$ with the generators $u_1, u_2, u_3, u_4$ corresponding to the vertices in the path $P_4$ with $u_i$ adjacent to $u_{i+1}$ for all $i$, and by the assumptions in the statement of the theorem $u_i u_j \gl u_j$ for all $u_i, u_j \in Y$. 

Next, for every pair $i, j \in \{1,2,3,4\}$ we fix a word $v_{i,j} \in A^*$ such that $v_{i,j}u_iu_j = u_j$.
This is possible since $u_iu_j \gl u_j$. 
Then for any non-empty word $w = u_{i_1}u_{i_2} \ldots u_{i_n} \in Y^*$ we define 
\[
\overline{w} = v_{i_n,2} v_{i_{n-1},i_{n}} v_{i_{n-2},i_{n-1}} \ldots v_{i_3,i_4} v_{i_2,i_3} v_{i_1,i_2} 
\]
and observe that 
\[
\overline{w} w u_2 =
v_{i_n,2} v_{i_{n-1},i_{n}} v_{i_{n-2},i_{n-1}} \ldots v_{i_3,i_4} v_{i_2,i_3} v_{i_1,i_2}
u_{i_1} u_{i_2} \ldots u_{i_n}
u_2 
= v_{i_n,2} u_{i_n} u_2 
= u_2
\]
in $M$. 
 
By Theorem~\ref{thm: nondecidability} 
there are two fixed rational subsets $Q, R \subseteq Y^*$, with $Q$ not containing the empty word, such that it is undecidable whether there exists a pair $(x,y) \in Q \times R$ satisfying $x(u_2)^i = y$ in $M$, for a given $i \in \N$.
Since $Y \subseteq A$ it follows that $Q, R$ are also rational subsets of $A^*$ satisfying this property.  

Set 
$\overline{Q} = \{ \overline{w} : w \in Q \}$. 
We claim that $\overline{Q}$ is a rational subset of $A^*$. 
To see this note that
$Q$ is a rational subset of $A^*$ not containing the empty word, from which it follows that 
$Qu_2$ is a rational subset of $A^*$ in which every word has length at least two. 
For every word $w = u_{i_1} u_{i_2} \ldots u_{i_{k-1}} u_{i_k}\in A^*$ of length at least $2$ define  
\[
\phi(u_{i_1} u_{i_2} \ldots u_{i_{k-1}} u_{i_k}) = 
v_{i_{k-1},i_k} 
v_{i_{k-2},i_{k-1}} 
\ldots 
v_{i_2,i_3} 
v_{i_1,i_2}.  
\]
Since $Qu_2$ is a rational subset of $A^*$ it follows from Lemma~\ref{lem_PresRegular} that $\phi(Qu_2)$ is a rational subset of $A^*$. 
But $\overline{Q} = \phi(Qu_2)$, and hence $\overline{Q}$ is a rational subset of $A^*$, completing the proof of the claim.

Hence we have a left-cancellative monoid $M$ with finite generating set $A$,  
rational languages $Q, R, \overline{Q} \subseteq A^*$,   
a surjective map $\alpha: Q \longrightarrow \overline{Q}$ with $\alpha(x) = \overline{x}$, and a fixed word $u_2 \in A^*$, satisfying the following properties:
\begin{itemize}
\item[(i)] 
$\overline{x} x u_2 = u_2$ in $M$, for all $x \in Q$,
\item[(ii)] 
it is undecidable whether there exists a pair $(x,y) \in Q \times R$ satisfying $x(u_2)^i = y$ in $M$, for a given $i \in \N$.
\end{itemize}
Then by Theorem~\ref{thm: general theorem undecidability} 
the rational language $\overline{Q}R \subseteq A^*$ defines a fixed rational subset of $M$ in which membership is undecidable.
\end{proof}

A natural question is the following 

\begin{question} 
If $m=1$ or $n=1$ then does $\M=\Mpres{a, b}{(b a^n)^m(a^n b)^ma = a}$ have decidable rational subset membership problem?

We know from above that the group with the same presentation does have decidable rational subset membership problem when $m=1$ or $n=1$.   
  \end{question}

\begin{remark}\label{rem_MoreNonSubspecialExamples} 
We have seen that for all $m,n \geq 2$ the monoid 
$\M$
contains a fixed rational subset in which membership is undecidable.  
There are easy modifications of this example that can be used to obtain monoids of the form 
$\Mpres{a,b}{bUa=aVa}$ with the same property. One simple way to do this would be to replace $a$ by $aaa$ and $b$ by $bbb$ to obtain the family of monoids     
\[
\Mpres{a, b}{[bbb (aaa)^n]^m[(aaa)^n bbb]^m aaa = aaa}. 
\]
It is straightforward to show that the submonoid of this monoid generated by $\{aaa,bbb
\}$ is isomorphic to $\M$, and hence for every $m,n \geq 2$ this monoid contains a fixed rational subset in which membership is undecidable, and this monoid has the form  $\Mpres{a,b}{bUa=aVa}$.  
  \end{remark}

\begin{remark} 
In a very similar way to Example~\ref{ex_compressSpecial} the monadic examples constructed above can be combined with the theory of compression 
\cite{Lallement1974, Kobayashi2000}
to obtain many more examples of non-subspecial monoids with undecidable rational subset membership problem. Indeed, it is not difficult to show that one-step compression by an overlap-free word (in the sense of Kobayashi \cite{Kobayashi2000}) preserves the property of having decidable rational subset membership problem. So if a monoid compresses in finitely many steps to one of our monadic example above with undecidable rational subset membership problem then the original monoid will have the same property. To give an example, for $m,n \geq 1$ the monoid   
\[
T = \Mpres{x, y}{(xyyxy (xyxxy)^n)^m((xyxxy)^n xyyxy)^m  xyxxy=  xyxxy}
\]
is sealed by the self-overlap free word $xy$, and when we perform one-step compression with respect to this we obtain the monoid   
\[
\M= \Mpres{a, b}{(b a^n)^m(a^n b)^ma = a}. 
\]
It follows that for $m,n \geq 2$ the monoid $T$ above contains a fixed rational subset in which membership is undecidable. Clearly $T$ is not a subspecial monoid.   
There is evidence that the converse should be true, that is, a compressible one-relation monoid has decidable rational subset membership problem if and only if its full compression does.
Combining this with the comments above on subspecial monoids, the authors believe that the problem of classifying one-relation monoids with decidable rational subset membership problem should reduce to solving the problem for positive one-relator groups (for the subspecial case) and solving it for incompressible monoids (in the non-subspecial case).    
\end{remark}

\subsection{An application to the rational subset membership problem in groups}

Corollary~\ref{cor_T4EmbedsInGroup} shows that if $G$ is a finitely generated group which embeds the trace monoid $T(P_4)$ then $G$ contains a fixed rational subset in which rational subset membership is undecidable. 
Currently all known examples of one-relator groups with undecidable rational subset membership problem embed $A(P_4)$. This motivates the following question: 

 \begin{question} 
Is there a one-relator group which embeds the trace monoid $T(P_4)$ but does not embed $A(P_4)$?
  \end{question}
\noindent If such an example exists it would show that for one-relator groups embeddability of $A(P_4)$ is sufficient but not necessary for undecidability of the rational subset membership problem. 

In the following proposition we will give an example of a group $H$ with a finite subset $X$ such that $\Mgen{X}$ is isomorphic to the trace monoid $T(P_4)$ but $\Ggen{X}$ is not isomorphic to $A(P_4)$. In particular, since $H$ is a group embedding $T(P_4)$, it follows from Corollary~\ref{cor_T4EmbedsInGroup} that $H$ contains a rational subset in which membership is undecidable.  We do not know whether the group $H$ embeds $A(P_4)$. And the only way we know of proving that $H$ has undecidable rational subset membership problem is by using the fact that it embeds $T(P_4)$. This shows that Corollary~\ref{cor_T4EmbedsInGroup} can be usefully applied to examples.   

\begin{prop}\label{prop_traceMonEmbeds} 
Let 
\[
H = \Gpres{x,y,z,t}{tx=xt, xz=zx, zy=yz, y^2x = xy},
\]
and set $X = \{ t,x,z,xy\}$. Then  
\begin{itemize} 
\item[(i)] $\Mgen{X}$ is isomorphic to the trace monoid $T(P_4)$, while 
\item[(ii)] $\Ggen{X}$ is not isomorphic to the right-angled Artin group $A(P_4)$.  
  \end{itemize}
In particular the group $H$ contains a rational subset in which membership is undecidable.    
  \end{prop}
\begin{proof} 
(i)
Let $K = \Gpres{x,y,z}{xz=zx, yz=zy, y^2x = xy}$. Then $T = \Mpres{x,y,z}{xz=zx, yz=zy, y^2x = xy}$ naturally embeds in the group $K$.

To see this observe that with $K_1 = \Gpres{z}{} $ and $K_2 = \Gpres{x,y}{y^2x = xy} $ we have   
\[
K = K_1 \times K_2 \cong \mathbb{Z} \times BS(1,2), 
\]
while setting $T_1 = \Mpres{z}{}$ and $T_2 = \Mpres{x,y}{y^2x = xy}$ we have 
\[
T = T_1 \times T_2 \cong \mathbb{N}_0 \times  \Mpres{x,y}{y^2x = xy}. 
\] 
Both $T_1$ and $T_2$ are group embeddable; $T_1$ is free, and $T_2$ is left and right cycle-free, so is group-embeddable \cite[Theorem~5]{Adian1960}. Hence the natural maps $\phi_i: T_i \rightarrow K_i$ for $i=1,2$ define injective homomorphisms. But then it follows that the map $\phi \colon T_1 \times T_2 \rightarrow K_1 \times K_2 $ defined by $(t_1,t_2) \mapsto (\phi_1(t_1), \phi_2(t_2))$ defines an injective homomorphism from $T$ into $K$. Hence $T$ is group embeddable and thus $T$ embeds naturally into the group with the same presentation, that is, the identity map on $\{x,y,z\}$ defines an injective homomorphism from $T$ into $K$.     

Then $T = \Mpres{x,y,z}{xz=zx, yz=zy, y^2x = xy}$ is a finite complete presentation for the monoid $T$, and if we consider the submonoid of $T$ generated by $\{x, xy, z\}$ we see that the reduced form of any word in $\{x, xy, z\}^*$ belongs to the set $\{z\}^*\{x,xy\}^*$ since the rewrite rule $y^2x = xy$ can never be applied to a word in $\{x, xy, z\}^*$. We conclude that the submonoid of $T$, and hence also of $K$, generated by $\{x, xy, z\}$ is isomorphic to $\mathbb{N}_0 \times \{c,d\}^*$ i.e. it is isomorphic to the trace monoid $T(P_3)$ where the vertices of $P_3$ are $x, z, xy$ in that order (with $z$ being the middle vertex of the path).         

Now apply Theorem~\ref{SubmonoidsOfHNNExtensions} below to the HNN-extension $H$ of $K$ where   
\[
H =  \Gpres{x,y,z, t}{xz=zx, zy=yz, y^2x = xy, txt^{-1}=x}. 
\]
From the previous paragraph the submonoid of $K$ generated by $\{x, z, xy\}$ is isomorphic to the trace monoid $\Mpres{x,b,c}{xb=bx, bc=cb, cd=dc}$. 

Then by Theorem~\ref{SubmonoidsOfHNNExtensions} it follows that the submonoid of $H$ generated by $X = \{ t,x,z,xy\}$ is isomorphic to $\Mpres{t,x,b,c}{tx=xt, xb=bx, bc=cb, cd=dc} \cong T(P_4)$, as required.  

(ii) $\Ggen{X} \cong H$ which has abelianisation $\mathbb{Z}^3$ and hence cannot be isomorphic to $A(P_4)$ which has abelianisation $\mathbb{Z}^4$.     
\end{proof}

The proof of Proposition~\ref{prop_traceMonEmbeds} uses the following general result about submonoids of HNN extensions of groups, which is also of independent interest.   

\begin{theorem}\label{SubmonoidsOfHNNExtensions}
Let $G \cong \Gpres{B}{Q}$, let $A \subseteq B$, fix $a \in A$ and set 
\[
H = \Gpres{B,t}{Q, tat^{-1}=a}
\]  
which is an HNN-extension of $G$ with respect to the automorphism fixing the cyclic subgroup of $G$ generated by $a$. 
Let $M$ be the submonoid $\Mgen{A}$ of $G$ generated by $A$, and let $\Mpres{A}{R}$ be a presentation for $M$ with respect to the generating set $A$. 
Then 
\[
\Mgen{M \cup \{ t \}} \cong \Mpres{A,t}{R, at=ta}. 
\]   
\end{theorem}

Before giving the proof, we recall the normal form in HNN-extensions, which is an immediate consequence of Britton's Lemma (see \cite[Ch. IV]{lyndon1977combinatorial}): 

\begin{lemma}[cf. Lemma 6.2 in \cite{dolinka2021new}]\label{lem: HNN normal form}
An equality of two reduced forms
\[
g_0 t^{\varepsilon_1} g_1 t^{\varepsilon_2} \cdots t^{\varepsilon_n} g_n = h_0 t^{\delta_1} h_1 t^{\delta_2} \cdots t^{\delta_m} h_m
\]
holds in the $\operatorname{HNN}$-extension
$G*_{t, \Phi: N_1 \rightarrow N_2}$ if and only if $n = m, \varepsilon_i = \delta_i$ for all $1 \leq i \leq n$, and there exist $1 = \alpha_0, \alpha_1, \ldots, \alpha_n, \alpha_{n+1} = 1 \in N_1 \cup N_2$ such that for all
$0 \leq i \leq n$ we have $\alpha_i \in N_1$ if $\varepsilon_i = -1$, $\alpha_i \in N_2$  if $\varepsilon_i = 1$, and
\[
h_i = \alpha_i^{-1} g_i (t^{\varepsilon_{i+1}} \alpha_{i+1}t^{-\varepsilon_{i+1}}).
\]
\end{lemma}

\begin{proof}[Proof of Theorem \ref{SubmonoidsOfHNNExtensions}]

Set $N =  \Mpres{A,t}{R, at=ta}$. There is an obvious morphism~$j\colon N \rightarrow H$, given by $j(a) = a$ for any $a \in A$, and $j(t) = t$. 

Moreover, the image $j(N)$ in $H$ is equal to $\Mgen{M \cup \{ t \}}$. To prove the result, it is enough to show that $j$ is injective as well. Let 
\begin{equation}\label{equality in j(N)}
g_0 t g_1 t \cdots t g_n = h_0 t h_1 t \cdots t h_m
\end{equation}
be an equality of elements in $j(N) = \Mgen{M \cup \{ t \}}$, where $M = \Mgen{A}$; i.e. $g_i, h_j \in M$ for all $0 \leq i \leq n$, and $0 \leq j \leq m$.

Equality \eqref{equality in j(N)} holds in the group $H$ as well, so we can use Lemma \ref{lem: HNN normal form} for $G \cong \Gpres{B}{Q}$, $N_1 = N_2 = \langle a \rangle \cong \Z$, $H = G*_{t, \Phi: \langle a \rangle \overset{id}{\longrightarrow} \langle a \rangle}$, and we obtain: 
\begin{center}
$n = m$, and there are $1 = \alpha_0, \alpha_1, \ldots, \alpha_n, \alpha_{n+1} = 1 \in \langle a \rangle$ with $h_i = \alpha_i^{-1} g_i (t \alpha_{i+1}t^{-1})$.
\end{center}
Since $t \alpha t^{-1} = \alpha$ for any $\alpha \in \langle a \rangle$, we obtain $h_i = \alpha_i^{-1} g_i \alpha_{i+1}$ in $H$.

{\bf Claim.}
Let $\alpha_{i+1} = a^s$ for some $0 \leq i \leq n$. The following equalities hold in $N$:
\[
\begin{array}{lcll}
h_0 t h_1 t \cdots t h_i & = & g_0 t g_1 t \cdots t g_i \alpha_{i+1}, & \text{ if $s \geq 0$}; \\
h_0 t h_1 t \cdots t h_i \alpha_{i+1}^{-1} & = & g_0 t g_1 t \cdots t g_i, & \text{ if $s < 0$}.
\end{array}
\]

\begin{proof}[Proof of Claim]
From the equality $h_i = \alpha_i^{-1} g_i \alpha_{i+1}$ in $H$, for $i=0$ one obtains $h_0 = g_0 \alpha_1$. Multiplying by $t$ on the right and using $\alpha_1 t = t \alpha_1$ (because $\alpha_1 \in \langle a \rangle$) we obtain the desired conclusion for $i = 0$.

Assume the result holds for a given $i < n$. We want to show that it holds for $i + 1$ as well. Set $\alpha_{i+2} = a^s$. 
\begin{itemize}[partopsep=0pt]
\item[(1)] Assume first that $h_0 t h_1 t \cdots t h_i = g_0 t g_1 t \cdots t g_i \alpha_{i+1}$ in $N$. Multiplying by $th_{i + 1}$, and using the commutation $\alpha_{i+1} t = t \alpha_{i+1}$ we obtain:
\begin{equation}\label{eqn: alpha_i+1 positive}
h_0 t h_1 t \cdots t h_i th_{i + 1} = g_0 t g_1 t \cdots t g_i t \alpha_{i+1} h_{i + 1}
\end{equation}
\begin{itemize}[partopsep=0pt]
\item[(a)] If $s \geq 0$ one has 
$\alpha_{i+1} h_{i+1} = g_{i+1} \alpha_{i+2}$ in $N$, which substituting in Equation \eqref{eqn: alpha_i+1 positive}, gives the desired conclusion.
\item[(b)] If $s < 0$ one has 
$\alpha_{i+1} h_{i+1}\alpha_{i+2}^{-1} = g_{i+1}$ in $N$. Multiplying Equation \eqref{eqn: alpha_i+1 positive} by $\alpha_{i+2}^{-1}$ on the right and substituting $\alpha_{i+1} h_{i+1}\alpha_{i+2}^{-1}$ by $g_{i+1}$, we obtain again the desired equality for $i+1$.
\end{itemize}
\item[(2)] Assume now that $h_0 t h_1 t \cdots t h_i \alpha_{i+1}^{-1} = g_0 t g_1 t \cdots t g_i$ in $N$. Multiplying by $tg_{i + 1}$, and using the commutation $\alpha_{i+1}^{-1} t = t \alpha_{i+1}^{-1}$ we obtain:
\begin{equation}\label{eqn: alpha_i+1 negative}
h_0 t h_1 t \cdots t h_i t \alpha_{i+1}^{-1} g_{i + 1} = g_0 t g_1 t \cdots t g_i t g_{i + 1}
\end{equation}
\begin{itemize}[partopsep=0pt]
\item[(a)] If $s \geq 0$ one has 
$h_{i+1} = \alpha_{i+1}^{-1} g_{i+1} \alpha_{i+2}$ in $N$. Multiplying Equation \eqref{eqn: alpha_i+1 negative} by $\alpha_{i+2}$ on the right and substituting $\alpha_{i+1}^{-1} g_{i+1} \alpha_{i+2}$ by $h_{i+1}$, we obtain again the desired equality for $i+1$.
\item[(b)] If $s < 0$ one has 
$h_{i+1}\alpha_{i+2}^{-1} = \alpha_{i+1}^{-1} g_{i+1}$ in $N$,  which substituting in Equation \eqref{eqn: alpha_i+1 negative}, gives the desired conclusion.
\end{itemize}
\end{itemize}
This shows the proof of our claim.
\end{proof}
Now the proof of the theorem is obtained by taking $i = n$ in our claim.
\end{proof}

\section{The word problem for positive special inverse monoids}\label{Sec:Positive-inverse-monoids}

It is an open question whether the word problem is decidable for all one-relation inverse monoids $\Ipres{A}{w=1}$ where $w$ is a reduced word. This problem is important since if the answer is yes then it would solve positively the longstanding open question of whether all one-relation monoids have decidable word problem.  In particular, this question is open in the case when $w \in A^+$ is a positive word. Interesting examples of positive one-relator special inverse monoids with counter-intuitive behaviour have been studied in the literature e.g. the O'Hare monoid \cite{Margolis1987}, and subsequent simpler examples \cite{NybergBrodda2022b}. Given that the question for positive one-relator inverse monoids remains open, it is natural to consider what happens in the case of two positive relators.  
Further motivation for studying positive two relator inverse monoids comes from the fact that for any 
pair of positive words 
$u, v \in A^+$ there is an isomorphism 
\[
\Ipres{A}{uv^{-1}=1} \cong 
\Ipres{A,t}{ut=1, vt=1}
\]  
which is a positive two-relator inverse monoid, and it follows from \cite{Ivanov2001} that the word problem for one-relation monoids reduces to the word problem for inverse monoids of this form. 
So the word problem for one-relation monoids reduces to the word problem for positive two-relator inverse monoids of the above form. 
While that question remains open,  
in this section we will show that in general the word problem for two-relator inverse monoids is not decidable.

\begin{theorem}\label{thm_2RelatorInverseUndecWP}
There is a positive two-relator inverse monoid $\Ipres{A}{u=1, v=1}$, 
where $u, v \in A^+$,   
with an undecidable word problem. 
Furthermore, the example may be chosen to be E-unitary.  
\end{theorem}

\begin{remark} 
Given that the word problem for two-relator groups is open, and also for two-relator monoids is open, it is natural to ask whether the groups or monoids defined by the presentations given by Theorem~\ref{thm_2RelatorInverseUndecWP} have decidable word problem. It is a consequence of the proof of  Theorem~\ref{thm_2RelatorInverseUndecWP} that the corresponding groups are in fact isomorphic to one-relator and hence have decidable word problem by Magnus' theorem. It is less obvious whether the corresponding two relator special monoids have decidable word problem, but by Makanin \cite{Makanin1966, Makanin1966b} they have groups of units that are two-relator groups, and the word problem reduces to that of the group of units. So those monoids are likely to have decidable word problem, else we would have an example of a two-relator group with undecidable word problem (namely the group of units of the monoid), and the word problem for two-relator groups is a famous open problem, cf. \cite[Problem~9.29]{Kourovka2020}
  \end{remark}

The following standard lemma will be helpful in proving the main result of this section; for a proof see e.g. \cite[Corollary 3.2]{gray2020undecidability}.

\begin{lemma}\label{lem_reduce} 
Let $M = \Ipres{A}{R}$. 
If $xaa^{-1}y \in (A \cup A^{-1})^*$ is right invertible in $M$, where $a \in A \cup A^{-1}$ and
 $x, y \in (A \cup A^{-1})^*$, then $xaa^{-1}y=xy$ in $M$.   
  \end{lemma}

The main result of this section, Theorem~\ref{thm_2RelatorInverseUndecWP}, will follow from the following general result that shows how to encode the submonoid membership problem in any positive one-relator group into a positive two-relator inverse monoid.

\begin{theorem}\label{thm_PosTwoRelatorGeneral} 
Let $G$ be a positive one-relator group, and let $Q$ be any finitely generated submonoid of $G$. Then there exists a positive two-relator inverse monoid $M=\Ipres{A}{u=1, v=1}$, with $u, v \in A^+$, such 
the membership problem for $Q$ in $G$ reduces to 
the word problem of $M$. 
Furthermore, $M$ can be chosen to be E-unitary and have maximal group image isomorphic to  
$G \ast \mathbb{Z}$. 
\end{theorem}
\begin{proof} 
Let $M$ be the inverse monoid defined by the two-relator positive presentation  
\[
M = M_{G,X} = \Ipres{B,x,t}{
r = 1, \quad 
t z_1 s t\overline{z_1} s  \ldots  s t z_k s t\overline{z_k} s = 1
}
\]
given by Construction~\ref{construction_encoding}. 
We claim that if $M$ has decidable word problem then the membership problem for $Q$ within $G$ is decidable.     
Furthermore we shall show that $M$ is E-unitary and has maximal group image isomorphic to $G \ast \mathbb{Z}$.

To establish this,  
we will perform a series of Tietze transformations on the presentation. 
This will come in two stages. 
First we will apply a set of Tietze transformations to prove that  
$M \cong T$ where 
\[
T = \Ipres{B,x,t}{
r= 1, \quad 
(t z_1 t^{-1}) (t z_1^{-1} t^{-1}) = 1, \quad 
\ldots, \quad  
(t z_k t^{-1}) (t z_k^{-1} t^{-1}) = 1
}.
\]

Since $\Mpres{A}{q=1}$ is a group it follows that every letter from $a$ is invertible in this monoid. From this it follows that in the monoid $\Mpres{B,x,t}{r=1}$, where $r$ is the word defined above obtained from $q$ by replacing $a$ by $tx$ everywhere it appears, the element $tx$ is invertible. 
Indeed, the map $a \mapsto tx$ 
and identity on every other letter 
defines a homomorphism from $\Mpres{A}{q=1}$ to $\Mpres{B,x,t}{r=1}$ which must map units to units, and since $a$ is a unit in the former monoid, $tx$ is a unit in the latter monoid.

Since $tx$ is a unit in $\Mpres{B,x,t}{r=1}$ it follows that $tx$ is a unit in the inverse monoid $M$, since $M$ has $r=1$ as a defining relation.
The fact that $tx$ is invertible in $M$ implies $(tx)^{-1} = x^{-1} t^{-1}$ is also invertible in the inverse monoid $M$. 
The first letter of $s$ is $x$, so $t z_1 x$ is right invertible in $M$ since it is a prefix of the second defining relator, hence the product $t z_1 xx^{-1} t^{-1} $ of two right invertible elements is right invertible, hence by Lemma~\ref{lem_reduce} it is equal in $M$ to $t z_1 t^{-1} $. We conclude that $t z_1 t^{-1} $ is right invertible in $M$.

Now $w_1 \overline{w_1} = \overline{w_1} w_1 = 1$ in $\Mpres{A}{q=1}$ implies 
$z_1 \overline{z_1} = \overline{z_1} z_1 = 1$ in $\Mpres{B,x,t}{r=1}$ since the map between these monoids given by $a \mapsto xt$, and identity on all other generators, defines a homomorphism with $z_1$ the image of $w_1$ under this homomorphism, and       
$\overline{z_1}$ is the image of $\overline{w_1}$.   
Since $r=1$ in $M$ it then follows that $z_1 \overline{z_1} = \overline{z_1} z_1 = 1$ in $M$, hence $z_1^{-1} = \overline{z_1}$ in $M$.    
Since $t z_1 t^{-1} $ is right invertible in $M$ we have 
$
t z_1 t^{-1} t z_1^{-1} t^{-1} = 1  
$
in $M$, then since  $z_1^{-1} = \overline{z_1}$ in $M$, and $r=1$ in $M$, it follows that 
$
t z_1 t^{-1} r t \overline{z_1} t^{-1} r = 1  
$
in $M$. Since this word equals $1$ it is right invertible and hence by Lemma~\ref{lem_reduce} it is equal in $M$ to the word obtained by cancelling the $t^{-1}$ with the first letters of $r$. It follows that       
\[ 
t z_1 s t\overline{z_1} s = t z_1 t^{-1} r t \overline{z_1} t^{-1} r = 1  
\]
in $M$. Hence  
\[
t z_2 s t\overline{z_2} s  \ldots  s t z_k s t\overline{z_k} s =
t z_1 s t\overline{z_1} s  \ldots  s t z_k s t\overline{z_k} s = 1
\]
in $M$. It follows that $t z_2 x$ is right invertible since $s$ begins with $x$, and we can repeat the argument above to deduce that 
$t z_2 t^{-1} $ is right invertible in $M$ and  
\[
t z_2 s t\overline{z_2} s =
t z_2 t^{-1} r t \overline{z_2} t^{-1} r =
 1  
\]
in $M$. 
Repeating this for all $j$, we can prove that 
$t z_j t^{-1} $ is right invertible  in $M$, and we have  
\[
t z_j s t\overline{z_j} s = 
t z_j t^{-1} r t \overline{z_j} t^{-1} r =
1  
\]
in $M$. 
In particular since
$t z_j t^{-1} $ is right invertible  in $M$ 
this means that the defining relations   
$
(t z_j t^{-1}) (t z_j^{-1} t^{-1}) = 1  
$
from the presentation for $T$ all hold in $M$ for $1 \leq j \leq k$.     

Conversely, 
since $z_j \overline{z_j} = \overline{z_j} z_j = 1$ in $\Mpres{B,x,t}{r=1}$, 
and 
$r=1$ in $T$, it follows that $z_j^{-1} = \overline{z_j}$ in $T$ for all $1 \leq j \leq k$. It follows that 
\[
(t z_1 t^{-1})  (t\overline{z_1} t^{-1})  \ldots   (t z_k t^{-1})  (t\overline{z_k} t^{-1})  = 
(t z_1 t^{-1})  (tz_1^{-1} t^{-1})  \ldots   (t z_k t^{-1})  (tz_k^{-1} t^{-1}) =
1 
\] 
in $T$ which, since $r=1$, implies that   
\[
(t z_1 t^{-1}) r (t\overline{z_1} t^{-1}) r  \ldots  r (t z_k t^{-1}) r (t\overline{z_k} t^{-1}) r = 1 
\]  
holds in $T$.  
Since the word equals one and so right is invertible by Lemma~\ref{lem_reduce} we can reduce cancelling the $t^{-1} t$ in each subword $t^{-1}r$ and deduce that the following relation holds in $T$  
\[
t z_1 s t\overline{z_1} s  \ldots  s t z_k s t\overline{z_k} s = 1. 
\]
We have proved the second defining relator of $M$ holds in $T$, and the second defining relator of $T$ holds in $M$. Since the first defining relators are the same, and the generating sets are the same, we conclude that $M \cong T$. 
In fact we have proved that these are equivalent presentations in the sense that the identity map on $B \cup \{x,t\}$ induces an isomorphism between them.  

To complete the proof we now prove a second sequence of Tietze transformations to the presentation for $T$ to obtain a presentation to which the main construction from \cite{gray2020undecidability} can then be applied to finish the proof. 
To do this we first we add a redundant generator $a$ and relation $a=tx$ to obtain    
\[
\Ipres{B,a,x,t}{
r= 1, \quad 
(t z_1 t^{-1}) (t z_1^{-1} t^{-1}) = 1, \;
\ldots, \;  
(t z_k t^{-1}) (t z_k^{-1} t^{-1}) = 1, \;
a=tx
}. 
\]
Next in every $z_i$ and in $r$ we can replace every occurrence of $xt$ by the letter $a$ obtaining, since $A = B \cup \{a\}$, the presentation     
\[
\Ipres{A,x,t}{
q= 1, \; 
(t w_1 t^{-1}) (t w_1^{-1} t^{-1}) = 1, \; 
\ldots, \; 
(t w_k t^{-1}) (t w_k^{-1} t^{-1}) = 1, \;
a=tx
}. 
\]
By assumption the relations $aa^{-1}=1$ and $a^{-1}a=1$ are both consequences of $q=1$ hence from $a=tx$ we can deduce that $1 = a^{-1}a = a^{-1}tx$ in this inverse monoid. Hence we can deduce $x = t^{-1}a$ in this inverse monoid, giving the following presentation      
\[
\bigl\langle A,x,t 
\:|\: 
q= 1, \quad 
(t w_1 t^{-1}) (t w_1^{-1} t^{-1}) = 1, \quad 
\ldots, \quad  
(t w_k t^{-1}) (t w_k^{-1} t^{-1}) = 1, \quad
\]
\[
a=tx, \quad 
x = t^{-1}a
\bigr\rangle.
\]
Now the relation 
$(t w_1 t^{-1}) (t w_1^{-1} t^{-1}) = 1$ implies that $tt^{-1}=1$ in this inverse monoid, and so from $x = t^{-1}a$ we can deduce the relation $tx = tt^{-1}a = a$. Since this relation is a consequence of the others we can now remove it obtaining the presentation   
 \[
\Ipres{A,x,t}{
q= 1, \; 
(t w_1 t^{-1}) (t w_1^{-1} t^{-1}) = 1, \; 
\ldots, \quad  
(t w_k t^{-1}) (t w_k^{-1} t^{-1}) = 1, \;
x = t^{-1}a
}
\]
for the same monoid. Now $x$ is a redundant generator which can be removed. This has no impact on the other relations since neither $x$ nor $x^{-1}$ appears in any of those words. Hence we arrive at the following presentation  
\[
\Ipres{A,t}{
q= 1, \quad 
(t w_1 t^{-1}) (t w_1^{-1} t^{-1}) = 1, \quad 
\ldots, \quad  
(t w_k t^{-1}) (t w_k^{-1} t^{-1}) = 1 \quad
}
\]
for $T$. Finally, by assumption the relations $cc^{-1}=1$ and $c^{-1}c=1$ for $c \in A$ are all consequences of $q=1$ so we can add them to obtain      
\[
\bigl\langle  
A,t
\:|\: 
q= 1, \; 
cc^{-1}=1, \ c^{-1}c=1 \; (c \in A), \quad 
(t w_1 t^{-1}) (t w_1^{-1} t^{-1}) = 1, \quad 
\]
\[
\ldots, \quad  
(t w_k t^{-1}) (t w_k^{-1} t^{-1}) = 1 
\bigl\rangle. 
\]
Now it follows from \cite[Theorem~3.8]{gray2020undecidability} that $M \cong T$ is an E-unitary inverse monoid, and 
if $M \cong T$ has decidable word problem then the membership problem for $Q = \Mgen{w_1, \ldots, w_k} \leq G$ within $G$ is decidable.
The fact that \cite[Theorem~3.8]{gray2020undecidability} applies to the presentation above follows from \cite[Lemma~3.3]{gray2020undecidability} (see line three of the proof of Theorem~3.8 in \cite{gray2020undecidability}). 
The maximal group image of $T$ is isomorphic to   
\[
\bigl\langle  
A,t
\:|\: 
q= 1, \; 
cc^{-1}=1, \ c^{-1}c=1 \; (c \in A), \quad 
(t w_1 t^{-1}) (t w_1^{-1} t^{-1}) = 1, \quad 
\]
\[
\ldots, \quad  
(t w_k t^{-1}) (t w_k^{-1} t^{-1}) = 1 \quad
\bigl\rangle  
\]
which is isomorphic to $\Gpres{A}{q=1} \ast FG(t) \cong G \ast \mathbb{Z}$, completing the proof of the theorem. 
\end{proof}

By combining this general theorem with the earlier results form this paper we can now prove the main result of this section. 

\begin{proof}[Proof of Thoerem~\ref{thm_2RelatorInverseUndecWP}]
Let $G$ be a positive one-relator group and let $Q$ be a finitely generated submonoid in which the membership problem is undecidable; such examples exist by Theorem~\ref{thm_mainPositive}. Then by Theorem~\ref{thm_PosTwoRelatorGeneral} there is a $E$-unitary positive two-relator inverse monoid 
$M=\Ipres{A}{u=1, v=1}$, with $u, v \in A^+$, such that decidability of the word problem for $M$ would imply decidability of the membership problem for $Q$ in $G$. Hence $M$ is an $E$-unitary positive two-relator inverse monoid with 
undecidable word problem.  
  \end{proof}

As a second application we now show how Theorem~\ref{thm_PosTwoRelatorGeneral} can be used to relate the membership problem in positive one-relator groups to the prefix membership problem in positive two-relator groups. 
It remains an open question whether there is a one-relator group presentation with cyclically reduced relator and undecidable prefix membership problem \cite[Question 13.10]{Bestvina2004}. In particular, it is not known whether there are positive one-relator groups with undecidable prefix membership problem.
Using Theorem~\ref{thm_2RelatorInverseUndecWP} and its proof we will show that if one allows two relators then examples with positive relators do exist. The following result is also the key ingredient used in the proof of Theorem~\ref{Thm:exists-quasi-positive-undec-prefix} above that shows there are quasi-positive one-relator groups with undecidable prefix membership problem.
 
\begin{theorem}
\label{thm_TwoRelator} 
Let $G$ be a positive one-relator group, and let $Q$ be any finitely generated submonoid of $G$. 
Then there is a positive two-relator group $H=\Gpres{A}{u=1, v=1}$, where $u, v \in A^+$, 
such that the membership problem for $Q$ in $G$ reduces to the prefix membership problem for $H$. 
Furthermore the group $H$ may be chosen such that the identity map on $A$ induces an isomorphism 
$\Gpres{A}{u=1, v=1} \rightarrow \Gpres{A}{u=1}$, and such that $H \cong G \ast \mathbb{Z}$.   
\end{theorem}
\begin{proof} 
Let $K$ be the group defined by the two-relator positive presentation  
\[
K = H_{G,X} =  
\Gpres{B,x,t}{
r = 1, \quad 
t z_1 s t\overline{z_1} s  \ldots  s t z_k s t\overline{z_k} s = 1
}
\]
given by Construction~\ref{construction_encoding}.
We claim if the prefix membership problem for the positive two-relator group $K$ is decidable then the membership problem for $Q$ in $G$ is decidable. 
To see this note that it is shown in the proof of  Theorem~\ref{thm_PosTwoRelatorGeneral} that 
\[
M = 
\Ipres{B,x,t}{
r = 1, \quad 
t z_1 s t\overline{z_1} s  \ldots  s t z_k s t\overline{z_k} s = 1
}
\]
is E-unitary and if $M$ has decidable word problem then   
the membership problem for $Q$ in $G$ is decidable.   
It is also shown there that the inverse monoid $M$ has maximal group image isomorphic to  
$G \ast \mathbb{Z}$ which is a one-relator group and thus    
has decidable word problem by Magnus' Theorem. 
It then follows from 
\cite[Theorem~3.3]{Ivanov2001} 
that $M$ has decidable word problem. 
But in Theorem~\ref{thm_PosTwoRelatorGeneral} it is proved that decidability of the word problem for $M$ implies decidability of the membership problem for $Q$ in $G$, hence the membership problem for $Q$ in $G$ is decidable, as claimed.  

In the proof of Theorem~\ref{thm_PosTwoRelatorGeneral} it is shown that the identity map on $B\cup\{x,t\}$ induces an isomorphism between the inverse monoids 
\[
M = \Ipres{B,x,t}{
r = 1, \quad 
t z_1 s t\overline{z_1} s  \ldots  s t z_k s t\overline{z_k} s = 1
}
\]
and 
\[
T = \Ipres{B,x,t}{
r= 1, \quad 
(t z_1 t^{-1}) (t z_1^{-1} t^{-1}) = 1, \quad 
\ldots, \quad  
(t z_k t^{-1}) (t z_k^{-1} t^{-1}) = 1
}.
\]
It follows that the identity map induces an isomorphism between the maximal group images of these inverse monoids 
\[
\Gpres{B,x,t}{
r = 1, \quad 
t z_1 s t\overline{z_1} s  \ldots  s t z_k s t\overline{z_k} s = 1
}
\rightarrow 
\Gpres{B,x,t}{r=1}. 
\]
(In general if the identity map induces an isomorphism $\Ipres{A}{R} \rightarrow \Ipres{A}{S}$ then for any additional set of relations $T$ the identity map will also induce an isomorphism between $\Ipres{A}{R \cup T} \rightarrow \Ipres{A}{S \cup T}$ since $S$ and $R$ are consequence of each other, and hence the same is true of $S \cup T$ and $R \cup T$. And so in particular it is true when $T$ is the set of relations $aa^{-1}= 1 = aa^{-1}$ added to get the maximal group image.)         
Now set $u \equiv r$ and $v \equiv t z_1 s t\overline{z_1} s  \ldots  s t z_k s t\overline{z_k} s $ and $Y = B \cup \{x,t\}$. Then the identity map on $Y$ induces isomorphisms 
\[
\Gpres{Y}{u=1,v=1} \rightarrow \Gpres{Y}{u=1}. 
\] 
It follows from Theorem~\ref{thm_PosTwoRelatorGeneral} that the maximal group image of $M$ is isomorphic to $G \ast \mathbb{Z}$. But the group  $\Gpres{Y}{vuv^{-1}=1}$ we have constructed is isomorphic to the maximal group image  $\Gpres{Y}{u=1,v=1}$ of $M$, so this completes the proof of the theorem.  
  \end{proof}

\begin{cor}\label{cor_TwoRelator} 
There is a positive two-relator group $\Gpres{A}{u=1, v=1}$, where $u, v \in A^+$, with an undecidable prefix membership problem. 
  \end{cor}
\begin{proof} 
This follows from Theorem~\ref{thm_TwoRelator} together with Theorem~\ref{thm_mainPositive} that shows there are positive one-relator groups containing finitely generated submonoids in which membership is undecidable. 
\end{proof}

\end{document}